\pdfoutput=1
\documentclass[11pt]{article}
\usepackage[a4paper,margin=2.6cm,top=2.6cm,bottom=3cm]{geometry}
\usepackage[T1]{fontenc}
\usepackage[utf8]{inputenc}
\usepackage{mathpazo}
\usepackage[scaled=0.92]{helvet}
\usepackage{amsmath,amssymb,amsfonts,amsthm,mathtools}
\usepackage{graphicx,booktabs,array,multirow,textcomp,placeins,url,cite}
\usepackage{etoolbox}
\usepackage[final]{microtype}
\usepackage{xcolor}
\definecolor{accent}{RGB}{24,40,72}
\definecolor{gold}{RGB}{242,183,5}
\definecolor{cream}{RGB}{255,252,240}
\definecolor{muted}{RGB}{100,100,96}
\usepackage{enumitem}
\usepackage[font=small,labelfont={sf,bf,color=accent},labelsep=period]{caption}
\usepackage[caption=false]{subfig}
\usepackage{titlesec}
\titleformat{\section}{\normalfont\sffamily\large\bfseries\color{accent}}{\thesection}{0.8em}{}
\titleformat{\subsection}{\normalfont\sffamily\normalsize\bfseries\color{accent}}{\thesubsection}{0.8em}{}
\titleformat{\subsubsection}[runin]{\normalfont\sffamily\normalsize\bfseries\color{accent}}{\thesubsubsection}{0.8em}{}
\titlespacing*{\section}{0pt}{2.2ex plus .5ex minus .2ex}{1.0ex plus .2ex}
\titlespacing*{\subsection}{0pt}{1.8ex plus .4ex minus .2ex}{0.8ex plus .2ex}
\titlespacing*{\subsubsection}{0pt}{1.2ex plus .4ex minus .2ex}{0.5em}
\usepackage{fancyhdr}
\fancypagestyle{plain}{\fancyhf{}\fancyfoot[C]{\footnotesize\sffamily\color{muted}\thepage}}
\usepackage[colorlinks,linkcolor=accent,citecolor=accent,urlcolor=accent,
  pdftitle={Data-proximal null-space networks for inverse problems},
  pdfauthor={Juergen Frikel, Simon Goeppel, Markus Haltmeier}]{hyperref}

\newtheoremstyle{thm}{0pt}{0pt}{\itshape}{}{\sffamily\bfseries\color{accent}}{.}{0.5em}{}
\newtheoremstyle{rem}{0pt}{0pt}{\normalfont}{}{\sffamily\bfseries\color{accent}}{.}{0.5em}{}
\theoremstyle{thm}
\newtheorem{theorem}{Theorem}[section]
\newtheorem{lemma}[theorem]{Lemma}
\newtheorem{corollary}[theorem]{Corollary}
\newtheorem{proposition}[theorem]{Proposition}
\theoremstyle{rem}
\newtheorem{definition}[theorem]{Definition}

\newtheorem{example}[theorem]{Example}
\newtheorem{remark}[theorem]{Remark}

\usepackage[most]{tcolorbox}
\tcbset{thmbox/.style={enhanced,breakable,sharp corners,boxrule=0pt,frame hidden,
  colback=cream,borderline west={2.5pt}{0pt}{gold},
  left=8pt,right=8pt,top=5pt,bottom=5pt,before skip=1.6ex,after skip=1.6ex}}
\tcolorboxenvironment{theorem}{thmbox}
\tcolorboxenvironment{proposition}{thmbox}
\tcolorboxenvironment{lemma}{thmbox}
\tcolorboxenvironment{corollary}{thmbox}
\tcolorboxenvironment{definition}{thmbox}
\tcolorboxenvironment{assumption}{thmbox}
\tcolorboxenvironment{remark}{thmbox,colback=white,borderline west={1.2pt}{0pt}{gold}}
\tcolorboxenvironment{example}{thmbox,colback=white,borderline west={1.2pt}{0pt}{gold}}

\renewenvironment{abstract}{\small\noindent\begin{minipage}{\linewidth}%
  {\sffamily\bfseries\color{accent}Abstract.}\ }{\end{minipage}}
\newenvironment{keywords}{\par\vspace{1.2ex}\small\noindent{\sffamily\bfseries\color{accent}Keywords.}\ }{\par}

\AtBeginEnvironment{remark}{\pushQED{\qed}}
\AtEndEnvironment{remark}{\popQED}
\AtBeginEnvironment{example}{\pushQED{\qed}}
\AtEndEnvironment{example}{\popQED}
\setlist[enumerate]{leftmargin=*,labelindent=0pt,labelsep=0.5em,itemsep=2pt,topsep=4pt}

\allowdisplaybreaks
\numberwithin{equation}{section}

\definecolor{mycolor}{RGB}{85,125,250}
\definecolor{gcolor}{RGB}{0.01,0.199,0.1}

\newcommand{\norm}[1]{\lVert#1\rVert}

\newcommand{\Ro}{R}
\newcommand{\Bo}{B}

\newcommand{\pen}{\mathcal P}

\newcommand{\risk}{\mathcal{E}}

\newcommand{\signal}{x}

\newcommand{\ydata}{y}

\newcommand{\ran}{\operatorname{ran}}
\renewcommand{\ker}{\operatorname{ker}}
\newcommand{\acqangles}{I}
\newcommand{\matid}{\mathbf{I}}
\newcommand{\radondir}{\omega}

\newcommand{\al}{\alpha}

\AtBeginDocument{

}

\DeclareMathOperator{\id}{id}
\DeclareMathOperator*{\argmin}{arg\,min}

\newcommand{\X}{X}
\newcommand{\Y}{Y}
\newcommand{\M}{X_0}
\newcommand{\YLS}{Y_{\mathrm{LS}}}

\definecolor{mathblue}{rgb}{0,0.28,0.55 }

\newcommand{\mpi}{+}

\newcommand{\WF}{\mathrm{WF}}

\newcommand{\noiselevel}{\delta}

\begin{document}
\thispagestyle{plain}
\begin{center}
{\sffamily\bfseries\LARGE\color{accent} Data-proximal null-space networks\\[0.3ex] for inverse problems\par}
\vspace{2.2ex}
{\large J\"urgen Frikel$^{1}$, Simon G\"oppel$^{2}$, and Markus Haltmeier$^{2}$\par}
\vspace{0.8ex}
{\small $^{1}$Department of Computer Science and Mathematics, OTH Regensburg, Galgenbergstra{\ss}e 32, 93053 Regensburg, Germany\\
$^{2}$Department of Mathematics, University of Innsbruck, Technikerstra{\ss}e 13, 6020 Innsbruck, Austria\\[0.4ex]
E-mail: \href{mailto:juergen.frikel@oth-regensburg.de}{\texttt{juergen.frikel@oth-regensburg.de}},
\href{mailto:markus.haltmeier@uibk.ac.at}{\texttt{markus.haltmeier@uibk.ac.at}}\par}
\vspace{1.2ex}
{\color{gold}\rule{0.35\linewidth}{1.2pt}}
\end{center}
\vspace{1ex}

\begin{abstract}
Inverse problems with incomplete data are severely ill-posed and require regularization. Learned
reconstructions incorporate prior information effectively, but their regularization properties are less well
understood than those of classical variational methods. We propose data-proximal null-space networks, which
post-process an initial reconstruction obtained by a classical regularization method and extend null-space
networks by allowing learned corrections that are not restricted to the null space, while bounding their effect
on the measured data by a prescribed tolerance. Combined with a convergent, possibly nonlinear regularization
method, these networks yield convergent regularization methods with convergence rates and data-proximity
bounds. A calmness-based analysis covers fixed networks with radial clipping even when the operator range is not
closed; their correction outside the null space vanishes in the limit, so
that they differ from null-space networks only at positive noise levels.
For the limited-angle Radon transform, exact data consistency preserves the visible singularities of
the initial reconstruction, whereas data proximity controls the visible error in a weak Sobolev norm
proportional to the data norm, but not its singularities. For compactly supported objects, null-space networks
act trivially, whereas data-proximal networks can still correct nearly invisible components. In limited-angle
computed tomography, data-proximal networks substantially improve on null-space networks and, unlike residual networks, guarantee a prescribed bound on the data
change.
\end{abstract}

\begin{keywords}
Regularization, computed tomography, limited-angle tomography, learned reconstruction, null-space network, data-proximal network, convergence analysis, data consistency, microlocal analysis
\end{keywords}
\vspace{2ex}
\section{Introduction}

We consider linear inverse problems
\begin{equation}\label{eq:ip}
    y^\delta = Ax_0 + e^\delta,
    \qquad \|e^\delta\|_{\Y}\leq\delta,
\end{equation}
where $A\colon\X\to\Y$ is a bounded linear operator between
separable Hilbert spaces. The objective is to reconstruct an
unknown signal $x_0$ from noisy measurements $y^\delta$ with additive
measurement noise  $e^\delta\in\Y$. Since inverse problems are typically
ill-posed, e.g., the forward operator \(A\) may not be injective, or its inverse on 
its range may be unbounded, reconstructions may be nonunique or sensitive 
to perturbations  in the data. 

The situation becomes even more challenging when measurements are
incomplete, as certain image features may no longer be stably recoverable from the available data.
To address these challenges, it is essential to incorporate additional prior information into
the reconstruction process \cite{engl1996regularization,scherzer2009variational,siltanen2003statistical}.
In recent years, learning-based reconstruction methods have emerged
as a powerful alternative to classical regularization approaches.
In this paper, we develop a learning-based method that combines the convergence
properties of classical regularization methods with learned corrections of an
initial reconstruction. Our approach is guided by the principle of microlocal consistency
\cite{Bubba_2019,andrade2022deepmicrolocal}, which aims to reliably recover
information contained in the measurements while using learned corrections to infer
missing structure. We implement this principle by explicitly controlling the effect of
these corrections on the measured data.

\subsection{Classical regularization and incomplete data}
\label{sec:intro-classical}

Classical regularization stabilizes the reconstruction by incorporating
prior information through penalties, spectral filters, or
iterative schemes
\cite{engl1996regularization,scherzer2009variational,ito2014inverse,mueller2012linear}.
A standard example is Tikhonov regularization, which minimizes $\|Ax-y^\delta\|_{\Y}^{2}+\alpha\|x\|_{\X}^{2}$ over $x\in\X$,
where $\alpha>0$ is a regularization parameter that balances
data fidelity and stability. A key property of classical regularization methods is that,
under suitable assumptions and parameter choice rules,
they are both data-proximal and convergent.
That is, the data residual $\|Ax_\alpha^\delta-y^\delta\|_{\Y}$ admits a quantitative
bound in terms of the noise level, while the reconstructions converge to a
specified solution as $\delta\to0$
\cite{scherzer2009variational}.
However, their manually chosen regularization functionals or
filter functions need not capture all complex structures
encountered in high-dimensional signals, such as images.

For incomplete measurements, microlocal analysis makes precise which image features can still be stably
recovered, by distinguishing visible singularities, which can be stably recovered from the measured data, from
invisible singularities, whose recovery requires additional prior information
\cite{quinto1993singularities,frikel2016generalized}.
It also characterizes which singularities can be recovered by classical reconstruction 
operators and which artifacts may arise from incomplete measurements
\cite{frikel2013artifacts,borg2018incompletect,hahn2016detectable}. These microlocal insights motivate an 
approach that preserves reliably recoverable image features while using learned 
priors to infer information not stably determined by the data.

\subsection{Learned reconstructions}
\label{sec:intro-learned}

Learning-based reconstruction methods have demonstrated substantial 
improvements in reconstruction quality compared with classical regularization
\cite{adler2017solving,arridge2019solving,li2020nett}.
Of particular interest are post-processing approaches
\cite{jin2017deep,antholzer2019deep},
which combine an initial regularization operator
$B_\alpha\colon\Y\to\X$ with a trainable image-space network
$f_\theta\colon\X\to\X$, parametrized by $\theta\in\Theta$, to give $\Ro_{\theta}^{\delta}= f_\theta\circ B_{\alpha(\delta)}$.
For example, $B_\alpha=(A^*A+\alpha\,\id)^{-1}A^*$
may be the Tikhonov reconstruction operator, and $f_\theta=\id+U_\theta$
an image-space residual network.

Typically, the network is trained on pairs
$(x_i,y_i^\delta)$, $i=1,\ldots,N$, by minimizing the
empirical reconstruction risk $\risk_N(\Ro_{\theta}^{\delta})=(1/N)\sum_{i=1}^{N}\|x_i-\Ro_{\theta}^{\delta}(y_i^\delta)\|_{\X}^{2}$~\cite{ratti2024sample}.
Other approaches incorporate the forward operator into
learned iterative schemes, variational network architectures,
or network cascades
\cite{adler2017solving,hammernik2018learning,kofler2018,hauptmann2018limitedview},
while learned regularization functionals provide another
way of combining data-driven priors with variational methods
\cite{li2020nett,alberti2021tikhonov,alberti2026sparsity}.
Despite their success, such learning-based methods do not
generally guarantee data consistency or convergence as a
regularization method merely by virtue of their architecture
or training objective. Learned reconstruction methods with such guarantees are surveyed in~\cite{mukherjee2023learned};
see also~\cite{ebner2024plug,aspri2020datadriven,aspri2020projection,hauptmann2025convergent}.

Null-space networks address the issue of data consistency by restricting
learned corrections to the null space of $A$ \cite{schwab2019deepnullspacelearning,Schwab2020}.
They take the form $f_\theta^{\mathrm{NSN}}= \id+P_{\ker(A)}U_\theta$ and satisfy $Af_\theta^{\mathrm{NSN}}=A$,
where $P_{\ker(A)}$ denotes the orthogonal projection onto
$\ker(A)$.
Hence, null-space networks preserve data fidelity by
construction. However, they leave the component in
$\ker(A)^\perp$ unchanged, thereby limiting their ability
to further improve reconstruction quality by enhancing
image structures and suppressing noise in this component.
Learned reconstructions can be unstable under small perturbations of the data~\cite{antun2020instabilities},
and instabilities and hallucinations have been traced to the null space of the forward
operator~\cite{gottschling2025troublesome}. A complementary adversarial analysis by Keltsch et al.~\cite{keltschAdversarialStability}
makes this limitation explicit. For null-space networks with a pseudoinverse
initialization, it shows that the error in the null-space complement is entirely determined
by the pseudoinverse, while data consistency alone does not control the
sensitivity of the learned null-space completion. This motivates allowing
controlled corrections that are not restricted to the null space, without interpreting data
consistency as a general robustness guarantee.

We extend the null-space networks to data-proximal (null-space) networks, which 
also allow corrections in the orthogonal complement of the null space while controlling their effect 
on the measured data. More precisely, we consider networks $f_{\theta,\beta}$ satisfying
\begin{equation}\label{eq:dp-intro}
    \|A(f_{\theta,\beta}(z)-z)\|_{\Y}\leq\beta\qquad\text{for all }z\in\X,
\end{equation}
which preserve the data proximity of an initial reconstruction $z$ up to an
additional error $\beta$. We write $z$ rather than $x$ for the input of a network, since in the
post-processing approach above the networks act on an initial reconstruction $z=B_\alpha(y^\delta)$. We call the left-hand side of \eqref{eq:dp-intro} the \emph{data change}; it
measures the effect of the learned correction of $z$ in the data space. The parameter $\beta$ thus acts as an
externally prescribed data-space correction budget, which bounds the data change but does not guarantee a
reduction of the initial data residual.
A similar principle, data proximity between two reconstructions with
different roles, couples the two stages of the complementary
$\ell^1$-TV reconstruction in~\cite{goeppel2024l1tv}.

\subsection{Microlocal consistency}

The microlocal insights discussed above provide an important
guiding principle for learning-based reconstruction:
visible singularities should be reliably reconstructed,
while learned priors can be used to recover invisible features.
This principle was explicitly exploited in \cite{Bubba_2019},
by combining model-based reconstruction of visible
shearlet coefficients with learned completion of
invisible coefficients.

Building on wavefront-set extraction \cite{andrade2019digitalwavefront}, the authors of 
\cite{andrade2022deepmicrolocal} explicitly addressed the preservation of visible 
singularities in learned reconstruction, describing such reconstructions as \emph{microlocally consistent}.
This approach incorporates microlocal information by explicitly
extracting wavefront sets and exploiting the canonical relation
of the forward operator. Related approaches have been developed 
for photoacoustic tomography \cite{pan2023invisible}, the recovery of
invisible interfaces \cite{rautio2024microlocalprior,karvonen2024tilt},
and microlocally motivated reconstruction architectures
and filters \cite{bubba2021psidonet,bubba2026filters}.

Null-space learning follows a related operator-based principle, although
membership in the exact null space is generally different from microlocal
invisibility. For objects supported in a bounded domain, the limited-angle Radon
transform is injective (Remark~\ref{rem:radon-compact-support}), and invisible
features correspond to small but nonzero singular values
\cite{davison1983illconditioned,louis1986incomplete}; such directions are
inaccessible to null-space networks but can be reached by data-proximal
corrections. For the continuous limited-angle Radon model considered below,
we use classical visibility properties to show that exact data consistency
implies preservation of visible singularities, without requiring explicit
wavefront-set extraction. Consequently, exact null-space post-processing
preserves the visible singularities of the initial reconstruction.
Data proximity does not generally guarantee singularity preservation. However, the data norm of a
function is proportional to a weak Sobolev norm of order $-1/2$ of its visible component. The
budget is therefore precisely a bound, in this weak norm, on the change that a data-proximal network makes to
the visible component of its input, while invisible components remain unconstrained. These observations provide a 
microlocal motivation and interpretation for the proposed regularization framework. 

\subsection{Contributions and outline}

This paper extends the regularization results for null-space networks established in 
\cite{schwab2019deepnullspacelearning} to data-proximal null-space networks.
In contrast to classical null-space networks, our architecture also allows learned corrections 
in the orthogonal complement $\ker(A)^\perp$, while explicitly controlling their effect
on the measured data through a prescribed data-space correction budget. The construction 
allows shared or independent networks for the two correction components.
In contrast to the range-nullspace decomposition of~\cite{chen2020deep}, which builds on the
pseudoinverse and leaves the data change of the correction in $\ker(A)^\perp$ uncontrolled, our construction starts
from a regularized initial reconstruction and bounds this data change by the prescribed budget
(Remark~\ref{rem:related-architectures}).
Data-proximal null-space networks and a first convergence analysis were
proposed in a preliminary version of this paper~\cite{goeppel2023dpnsn},
and a short conference paper~\cite{goeppel2025bvm} applied them to
limited-view CT. Compared with~\cite{goeppel2023dpnsn}, the present paper replaces
uniform Lipschitz assumptions by calmness at the exact solutions and adds the analysis of fixed networks with
radial clipping, the microlocal analysis for limited-angle tomography, and new experiments with explicit control
of projection leakage.

We prove that, under suitable assumptions and parameter choice rules, the proposed 
data-proximal null-space networks, combined with an initial regularization through 
$f_{\theta,\beta}\circ B_\alpha$, define convergent regularization methods for learned 
solution selections. In contrast to \cite{schwab2019deepnullspacelearning,Schwab2020},
our analysis does not require the initial regularization methods $(B_\alpha)_{\alpha>0}$ to 
be linear and thus also applies to suitable nonlinear variational regularization methods.
We further establish convergence rates and inheritance of data-proximity rates under 
additional assumptions. Our analysis uses the notion of calmness at the exact solutions 
and addresses the failure of uniform Lipschitz bounds in ill-posed settings.
For fixed networks with radial clipping, which rescales the learned
correction in $\ker(A)^\perp$ such that its effect on the data is at most $\beta$,
the ill-posedness enters only through a calmness quantity evaluated at the exact
solutions, and we obtain convergence as the budget vanishes, without an
additional coupling to the parameter of the convergent baseline regularization
(Corollary~\ref{cor:fixed-theta}).
The limiting solution selection
is that of the associated null-space network (Remark~\ref{rem:fixed-theta}). The theory thus guarantees
that the additional flexibility does not compromise convergence,
while its benefit at positive noise levels, where its effect on the
data is bounded by $\beta$, is assessed numerically.

Building on classical visibility theory, we show for limited-angle tomography that exact data
consistency preserves visible singularities, and that the budget bounds, in a weak Sobolev norm, the change of
the visible component. This gives data-proximal networks a precise microlocal interpretation.
Finally, we present numerical experiments in limited-angle CT comparing
data-proximal networks with residual and null-space networks using
Tikhonov and TV initializations. The comparison examines reconstruction
quality and the data change separately and accounts explicitly
for numerical projection leakage.
Data-proximal networks substantially improve on null-space networks and, unlike residual networks,
guarantee a prescribed bound on the data change; with Tikhonov initialization, the best of them nearly matches
residual networks.

This paper is organized as follows. Section~\ref{sec:background} recalls regularization methods,
data proximity, and null-space networks. Section~\ref{sec:DP} introduces data-proximal null-space networks
and establishes convergence and convergence rates. Section~\ref{sec:microlocal} introduces the limited-angle 
Radon transform and examines the implications of data consistency and data proximity
for visible singularities. Section~\ref{sec:application} describes the numerical implementation
and experiments. Finally, Section~\ref{sec:conclusion} summarizes the results and outlines further developments.

\section{Mathematical preliminaries}
\label{sec:background}

The inversion of \eqref{eq:ip} is unstable if $\ran(A)$ is non-closed and nonunique if $\ker(A)\neq\{0\}$.
Throughout Sections~\ref{sec:background} and~\ref{sec:DP}, $A\colon\X\to\Y$ is bounded and linear as in \eqref{eq:ip}.
In this section, we recall the basic concepts of regularization theory, data proximity, and null-space
networks, which provide the mathematical foundation for our proposed reconstruction framework.

\subsection{Regularization methods}

We first recall the classical concept of regularization,
which addresses nonuniqueness and instability in inverse problems.
To account for the possible nonexistence of exact solutions
and different solution selection principles, we consider
a generalized notion of least-squares right inverses,
which may also be nonlinear.

\begin{definition}[Least-squares right inverse]
\label{def:right}
Let $\YLS\coloneqq\ran(A)\oplus\ran(A)^\perp$ denote the set of data for which least-squares
solutions exist. A possibly nonlinear mapping $B\colon\YLS\to\X$ is called a least-squares (LS) right inverse of $A$ if
$A(B(y)) = P_{\overline{\ran(A)}}(y)$ for all $y \in \YLS$.
\end{definition}

Thus, $B$ maps each $y\in \YLS$ to a least-squares solution, i.e., $B(y)$ minimizes $\|Ax-y\|_{\Y}^2$ over $x\in\X$.
For exact data $y\in\ran(A)$, this reduces to $AB(y)=y$, which explains the term least-squares right inverse.
The Moore--Penrose inverse $A^\mpi$, with $D(A^\mpi)=\YLS$, is a particular LS right inverse of $A$ which selects the unique least-squares solution 
of minimal norm. Other selection principles among least-squares solutions give rise to different, possibly nonlinear LS right inverses.

It is well known that the Moore--Penrose inverse is discontinuous if $\ran(A)$ is non-closed
\cite{engl1996regularization}. The same holds for general LS right inverses, including nonlinear ones \cite{haltmeier2026dataconsistent}.
Hence, if $\ran(A)$ is non-closed, any LS right inverse requires regularization.

Regularization approximates a possibly discontinuous solution map $B$
by a family of continuous operators $(R_\alpha)_{\alpha>0}$.
Together with a suitable parameter choice rule
$\alpha^\star(\delta,y^\delta)$, these operators recover the selected
solutions as the noise level tends to zero,
cf. \cite{engl1996regularization,scherzer2009variational}.
This leads to the following definition.

\begin{definition}[Regularization method]
\label{def:reg}
Let $B \colon \YLS \to \X$ be a least-squares right
inverse of $A$ and let $\M\subseteq B(\ran(A))$ be an admissible class of exact solutions.
A pair $((R_\alpha)_{\alpha>0},\alpha^\star)$,
where $R_\alpha\colon\Y\to\X$ is continuous for every $\alpha>0$
and $\alpha^\star\colon(0,\infty)\times\Y\to(0,\infty)$,
is called a (convergent) regularization method for
\eqref{eq:ip} (based on $B$) over $\M$ if the following condition holds:
\begin{equation}\label{eq:reg}
\begin{aligned}
\forall\,x\in \M:\quad
\lim_{\delta\to0}\ \sup\Bigl\{
&\bigl\|x-
R_{\alpha^\star(\delta,y^\delta)}(y^\delta)
\bigr\|_{\X}
&\;:\;
y^\delta\in\Y,\;
\|Ax-y^\delta\|_{\Y}\leq\delta
\Bigr\}=0.
\end{aligned}
\end{equation}
In this case, $(R_\alpha)_{\alpha>0}$
is called a regularization of $B$, and $\alpha^\star$
is referred to as a parameter choice rule.
\end{definition}

Definition~\ref{def:reg} generalizes the classical notion of a regularization method
\cite{engl1996regularization,kirsch2021introduction} in two respects: the Moore--Penrose inverse is replaced
by a possibly nonlinear least-squares right inverse $B$, and convergence is required
only on a class $\M\subseteq B(\ran(A))$ of exact solutions, which may be a proper subset.
Since convergence is required for each $x\in\M$ separately, there is a largest such class
for given $((R_\alpha)_{\alpha>0},\alpha^\star)$. It may be a proper subset of $B(\ran(A))$,
for instance if $(R_\alpha)_{\alpha>0}$ regularizes a different least-squares right inverse
that coincides with $B$ only on part of $B(\ran(A))$. Regularization methods with a restricted
admissible set were considered in \cite{Schwab2020}, where, in addition, the parameter choice is
required to tend to zero.
The classical setting corresponds to $B=A^\mpi$ and $\M=\ker(A)^\perp$.
The noisy data $y^\delta$ may have arbitrary components in $\ran(A)^\perp$; only the
exact data are assumed to lie in $\ran(A)$, as in \cite{kirsch2021introduction}.
The definition in \cite{engl1996regularization} additionally admits exact data with a component in $\ran(A)^\perp$. For methods that depend on the data only
through their projection onto $\overline{\ran(A)}$, such as spectral filter and variational
methods, this yields an equivalent requirement; see also Section~\ref{sec:convergence}.

The following examples illustrate two common choices of the
least-squares solution map $B$: the Moore--Penrose inverse,
which selects minimum-norm solutions, and variational selections
based on $\pen$-minimizing LS solutions
\cite{engl1996regularization,scherzer2009variational}.

\begin{example}[MPI-based regularization methods]
\label{ex:mpi based regularization}
For $B=A^\mpi$, we take the largest possible class $\M=\ker(A)^\perp$, which
coincides with $B(\ran(A))=\ran(A^\mpi)$.
Spectral filter methods \cite{engl1996regularization,Schwab2020} approximate this solution map by
$\Bo_\alpha\coloneqq g_\alpha(A^*A)A^*$,
where the operator functions are defined by spectral calculus
and $g_\alpha$ approximates $\lambda\mapsto1/\lambda$.
More precisely, suppose
$g_\alpha:[0,\|A\|^2]\to\mathbb{R}$
is bounded and piecewise continuous for every $\alpha>0$,
and satisfies $\sup_{\alpha>0}\sup_{\lambda\in(0,\|A\|^2]}|\lambda g_\alpha(\lambda)|<\infty$ and $\lambda g_\alpha(\lambda)\to1$ as $\alpha\to0$ for $\lambda>0$.
Each $\Bo_\alpha$ is bounded and linear, hence continuous.
With an a priori parameter choice $\alpha^\star(\delta,y^\delta)\coloneqq\alpha(\delta)$,
which does not depend on $y^\delta$ and satisfies $\alpha(\delta)\to0$ and $\delta^2\|g_{\alpha(\delta)}\|_\infty\to0$,
the pair $((\Bo_\alpha)_{\alpha>0},\alpha^\star)$
defines a convergent regularization method based on $A^\mpi$ over $\M$
\cite{engl1996regularization}.
Two standard examples are Tikhonov regularization and the truncated singular value decomposition (TSVD).
The Tikhonov filter $g_\alpha(\lambda)=1/(\lambda+\alpha)$ satisfies $\|g_\alpha\|_\infty=1/\alpha$ and yields
$\Bo_\alpha=(A^*A+\alpha\id_{\X})^{-1}A^*$, which maps $y^\delta$ to the minimizer of the variational formulation
in Section~\ref{sec:intro-classical} and thus connects this method to the penalty-based approaches of the next example.
If $A$ admits a singular value decomposition $A=\sum_n\sigma_n\langle\,\cdot\,,v_n\rangle_{\X}u_n$ with nonzero
singular values $\sigma_n$, the TSVD filter $g_\alpha(\lambda)=\lambda^{-1}\chi_{[\alpha,\infty)}(\lambda)$ for
$\lambda>0$, with $g_\alpha(0)=0$, satisfies $\|g_\alpha\|_\infty\leq1/\alpha$ and yields
$\Bo_\alpha=\sum_{\sigma_n^2\geq\alpha}\sigma_n^{-1}\langle\,\cdot\,,u_n\rangle_{\Y}v_n$, which retains only
the terms of the singular value expansion of $A^\mpi$ with $\sigma_n^2\geq\alpha$.
\end{example}

\begin{example}[Variational regularization methods]
\label{ex:var-reg}
Variational regularization extends quadratic Tikhonov regularization by allowing more general penalty
functionals \cite{scherzer2009variational}. For a proper, convex, lower semicontinuous, and
coercive penalty $\pen\colon\X\to[0,\infty]$, regularized solutions
$\Bo_\alpha(y^\delta)\in\argmin_{x\in\X}\{\|Ax-y^\delta\|_{\Y}^2+\alpha\pen(x)\}$ exist for every $\alpha>0$. For
$y\in\YLS$, a $\pen$-minimizing LS solution minimizes $\pen$ among all LS solutions. If at least one LS solution has
finite $\pen$-value, such solutions exist, and for $\|y^\delta-y\|_{\Y}\leq\delta$, $\alpha(\delta)\to0$, and
$\delta^2/\alpha(\delta)\to0$, the regularized solutions converge weakly, along subsequences, to $\pen$-minimizing LS
solutions \cite{scherzer2009variational}. If $\pen$ is strongly convex and every $y\in\YLS$ admits an LS solution with finite $\pen$-value, the
$\pen$-minimizing LS solution is unique
and defines an LS right inverse $B$, each $\Bo_\alpha$ is single-valued and Lipschitz continuous, and the regularized
solutions converge strongly to $B(y)$, since strong convexity turns the convergence of the penalty values into norm
convergence. Hence $((\Bo_\alpha)_{\alpha>0},\alpha^\star)$ is a convergent regularization
method based on $B$ over the largest possible class $\M=B(\ran(A))$. This applies to the strongly convex
total-variation penalty $\pen_\varepsilon\coloneqq\mathrm{TV}+(\varepsilon/2)\|\cdot\|_{\X}^2$ with $\varepsilon>0$
\cite{rudin1992nonlinear}. The case $\varepsilon=0$, in which minimizers need not be unique,
can be treated within a set-valued framework \cite{scherzer2009variational}; for simplicity, we restrict ourselves to the
single-valued case.
\end{example}

\subsection{Convergence rates and data proximity}

Data proximity is a natural requirement for regularization
methods, since the exact solution fits the measured data
up to the noise level. We formalize this property through
a quantitative bound on the data residual and later
examine its preservation under data-proximal network
post-processing.
The following definition treats convergence rates for the reconstruction error in $\X$
and data-proximity rates for the data residual in $\Y$ in parallel.

\begin{definition}[Convergence and data-proximity rates]
\label{def:rates}\label{def:rate}\label{def:data-proximity}
Let $((\Ro_\al)_{\al>0},\al^\star)$ be a regularization method for \eqref{eq:ip} over $\M$,
and let $r,s\in(0,1]$. For $\signal\in\M$ and $\ydata^\noiselevel\in\Y$ with
$\norm{\ydata^\noiselevel-A\signal}_{\Y}\leq\delta$, write
$\signal_\delta\coloneqq\Ro_{\al^\star(\delta,\ydata^\noiselevel)}(\ydata^\noiselevel)$.
\begin{enumerate}[label=(\alph*)]
\item The regularization method $((\Ro_\al)_{\al>0},\al^\star)$ has
\emph{convergence rate} $s$ if there exist $C,\delta_0>0$ such that
$\norm{\signal-\signal_\delta}_{\X}\leq C\delta^s$ for all $\signal\in\M$,
$0<\delta\leq\delta_0$, and all such $\ydata^\noiselevel$.
\item The regularization method $((\Ro_\al)_{\al>0},\al^\star)$ has
\emph{data-proximity rate} $r$ if there exist $\tau,\delta_0>0$ such that
$\norm{\ydata^\noiselevel-A\signal_\delta}_{\Y}\leq\tau\delta^r$ for all $\signal\in\M$,
$0<\delta\leq\delta_0$, and all such $\ydata^\noiselevel$. In this case, the method is also
called \emph{rate-$r$ data-proximal}.
\end{enumerate}
If the bounds in (a) or (b) hold for all $\signal$ in a subset $\X_s\subseteq\M$, we say that the method has
the respective rate over $\X_s$.
\end{definition}

Convergence of a regularization method implies that the
data residual vanishes as $\delta\to0$, but does not
necessarily guarantee a data-proximity rate $r$ for any
$r\in(0,1]$.
For $r=1$, the residual is bounded by a constant multiple
of the noise level; $0<r<1$ allows a slower decay.
The constants $\tau$ and $\delta_0$ are uniform over
$\M$ and all admissible noisy data.

Many classical regularization methods satisfy such bounds
under suitable parameter choice rules.
In particular, Morozov's discrepancy principle directly
controls the residual in terms of the noise level
\cite{engl1996regularization}.
There, however, the residual serves as a criterion for choosing the regularization
parameter. Here, data proximity is considered as a property of the reconstruction method
itself, which is preserved under post-processing (Lemma~\ref{lem:dp-proximity}).
For suitable a priori parameter choices, data-proximity rate $1$
also holds for spectral filter methods under
source conditions
$\signal\in\ran((A^*A)^\mu)$, $\mu>0$,
provided the filters have sufficient qualification.
Similarly, variational regularization admits data-proximity rate $1$
under suitable source conditions,
such as
$\partial\pen(\signal)\cap\ran(A^*)\neq\varnothing$ for $\signal\in\M$.
Uniform estimates over $\M$ require corresponding
uniform bounds on the source elements.

The following example shows that spectral filtering
methods can achieve data-proximity rate $r$ for any
$r\in(0,1)$, even without additional source conditions.

\begin{example}[Data-proximity rates without source conditions]
\label{ex:regularization-DP}

Consider a spectral filtering regularization method $((\Bo_\al)_{\al>0},\al^\star)$ (cf. Example 
\ref{ex:mpi based regularization}), where $\Bo_\al=g_\al(A^*A)A^*$. Assume that the filters satisfy 
the standard regularization conditions, in particular \( \sup_{\al>0}\sup_{\lambda\in[0,\|A\|^2]}
|\lambda g_\al(\lambda)|<\infty\) and \(\norm{g_\al}_\infty\leq C/\al\),
and have qualification at least $1/2$, i.e., \(\sup_{\lambda\in[0,\|A\|^2]} \sqrt{\lambda}\,
|1-\lambda g_\al(\lambda)| \leq C\sqrt{\al};\)
see \cite{engl1996regularization}. Then, for $\norm{A\signal-\ydata^\noiselevel}_{\Y}\leq\delta$,
\begin{align*}
	\norm{A\Bo_\al\ydata^\noiselevel-\ydata^\noiselevel}_{\Y} &= \norm{ (Ag_\al(A^*A)A^*-\id)\ydata^\noiselevel }_{\Y} \\
	&\leq \norm{ (Ag_\al(A^*A)A^*-\id) (A\signal-\ydata^\noiselevel) }_{\Y} + \norm{ A(g_\al(A^*A)A^*A-\id)\signal }_{\Y} \\
	&\leq \norm{Ag_\al(A^*A)A^*-\id}\,\delta + \norm{A(g_\al(A^*A)A^*A-\id)}\,\norm{\signal}_{\X} \\
	&\leq C_0\delta+C_{1/2}\sqrt{\al}\, \norm{\signal}_{\X}.
\end{align*}
Here, the last inequality follows from the uniform boundedness of the filter residual operators and
the qualification assumption. Consequently, for any $r\in(0,1)$ and $R>0$, choosing $\al^\star(\delta)\asymp\delta^{2r}$
yields data-proximity rate $r$ uniformly on $\{\signal\in \ker(A)^\perp: \norm{\signal}_{\X}\leq R\}$
for sufficiently small $\delta$. Moreover, this parameter choice satisfies
$\al^\star(\delta)\to0$ and $\delta^2/\al^\star(\delta)\to0$, and hence defines a convergent regularization method
under the stated filter assumptions.
\end{example}

\begin{remark}[Rates in $\X$ and in $\Y$]
\label{rem:rates-X-Y}
Parts (a) and (b) of Definition~\ref{def:rates} measure the reconstruction
error in $\X$ and the data residual in $\Y$, respectively. Since $A$ is bounded,
convergence rate $s$ implies data-proximity rate $s$. Indeed, for $0<\delta\leq\min\{\delta_0,1\}$, we have
\[
\norm{\ydata^\noiselevel-A\Ro_{\al^\star(\delta,\ydata^\noiselevel)}(\ydata^\noiselevel)}_{\Y}
\leq \norm{A}\,\norm{\signal-\Ro_{\al^\star(\delta,\ydata^\noiselevel)}(\ydata^\noiselevel)}_{\X}+\delta
\leq (\norm{A}\,C+1)\,\delta^s .
\]
This bound is generally pessimistic. For spectral filter methods and exact data
$\ydata=A\signal$, the residual is $A(g_\al(A^*A)A^*A-\id)\signal$, so that the data residual
gains an additional order $1/2$ compared with the reconstruction error. Under the source
condition $\signal\in\ran((A^*A)^\mu)$ and for filters of sufficient qualification, the
reconstruction error is of order $\al^\mu$, whereas the residual is of order $\al^{\mu+1/2}$.
With $\al\asymp\delta^{2/(2\mu+1)}$, this yields convergence rate $2\mu/(2\mu+1)<1$ but
data-proximity rate $1$.
Moreover, data-proximity rates do not require source conditions. By
Example~\ref{ex:regularization-DP}, spectral filter methods have data-proximity rate $r$
for every $r\in(0,1)$ over bounded subsets of $\ker(A)^\perp$, whereas convergence rates in
$\X$ over such sets require source conditions and are impossible if $\ran(A)$ is not closed
\cite{engl1996regularization,burger2004convergence,grasmair2011necessary}. Data-proximity rates are therefore a weaker requirement. In contrast to
convergence rates, they concern the data residual, which can be evaluated for every reconstruction, and they are
preserved under post-processing (Lemma~\ref{lem:dp-proximity}).
\end{remark}

\subsection{Null-space networks}
\label{sec:nullspace}

Classical regularization methods select solutions according to prescribed mathematical criteria. For example, 
spectral filter methods approximate the Moore--Penrose inverse $A^\mpi$, which selects the minimum-norm 
least-squares solution, whereas variational methods select penalty-minimizing solutions. Null-space networks 
\cite{schwab2019deepnullspacelearning} allow this solution selection to be adapted through learned corrections 
while preserving data consistency.

Here and below, linear operators are applied without parentheses, $Lz$. Accordingly, for a linear
operator $L$ and a map $U$, we write $LU\coloneqq L\circ U$, whereas compositions with a possibly nonlinear map on
the left are denoted by $\circ$.
Consider a continuous network $f_\theta\colon\X\to\X$ applied to an initial
regularization method $\Bo_\al$; see Section~\ref{sec:intro-learned}. If $\Bo_\al$ converges to an LS right
inverse $B$ under a suitable parameter choice, then $f_\theta\circ\Bo_\al$ converges to
$f_\theta\circ B$. This composition need not be an LS right inverse, since in general
$Af_\theta(By)\neq y$ for $y\in\ran(A)$. Null-space networks are precisely those
networks that leave the data unchanged, $Af_\theta=A$.

\begin{definition}[Null-space networks and post-processing]
\label{def:nsn}
Let $(U_\theta)_{\theta\in\Theta}$ be a family of Lipschitz continuous mappings $U_\theta\colon\X\to\X$.
Then, for $\theta\in\Theta$, the mappings $f_\theta^{\mathrm{NSN}}\colon\X\to\X$ defined by
$f_\theta^{\mathrm{NSN}}\coloneqq\id+P_{\ker(A)}U_\theta$
are called null-space networks for $A$. For an initial regularization $(\Bo_\al)_{\al>0}$, we call
$\Ro_\al\coloneqq f_\theta^{\mathrm{NSN}}\circ\Bo_\al$ the null-space post-processing.
\end{definition}

By construction, every null-space network satisfies $Af_\theta^{\mathrm{NSN}}=A$ and thus preserves data consistency
and the data proximity of any initial reconstruction.
Conversely, every Lipschitz continuous map $f\colon\X\to\X$ with $Af=A$ is of this form
with $U=f-\id$, since $A(f-\id)=0$ implies $f-\id=P_{\ker(A)}(f-\id)$. The residual form in
Definition~\ref{def:nsn} is therefore not a restriction but a parametrization of all
data-consistent maps, and $f_\theta^{\mathrm{NSN}}$ depends on $\theta$ only through
the trainable correction $U_\theta$.
In contrast, standard residual networks $\id+U_\theta$
do not generally satisfy this property.

Moreover, if $B$ is a least-squares right inverse of $A$, then $f_\theta^{\mathrm{NSN}}\circ B$ is also a least-squares right inverse.
If $((\Bo_\al)_{\al>0},\al^\star)$ is a convergent regularization method based on $B$, then, for fixed $\theta$, 
the null-space post-processing $\Ro_\al=f_\theta^{\mathrm{NSN}}\circ\Bo_\al$
defines a convergent regularization method based on $f_\theta^{\mathrm{NSN}}\circ B$
over $f_\theta^{\mathrm{NSN}}(\M)$; this is the special case $V_\theta=0$ of
Theorem~\ref{thm:convergence}. For $B=A^\mpi$ and $\M=\ker(A)^\perp$, the class of exact
solutions is $f_\theta^{\mathrm{NSN}}(\ker(A)^\perp)$.
These regularization properties were established for Moore--Penrose-based regularization and null-space
networks in \cite{schwab2019deepnullspacelearning}.  Furthermore, null-space networks preserve the data
residual of the initial reconstruction exactly, since \(A\Ro_\al(\ydata^\noiselevel) = A\Bo_\al(\ydata^\noiselevel).\)
Consequently, if $\Bo_\al$ is rate-$r$ data-proximal, then so is $\Ro_\al$.

\begin{remark}[Computing the null-space projection]
\label{rem:projection-layer}
A key component of null-space networks is the orthogonal projection $P_{\ker(A)}$, which maps
$z\in\X$ to the element of $\ker(A)$ closest to $z$. For some operators, it is explicit; for the
subsampled discrete Fourier transform $A=\mathcal{S}_J\mathcal{F}$ with unitary $\mathcal{F}$ and
frequency selection $\mathcal{S}_J$, for instance,
$P_{\ker(A)}=\mathcal{F}^*(\id-\mathcal{S}_J^*\mathcal{S}_J)\mathcal{F}$. In finite dimensions,
$P_{\ker(A)}=\id-QQ^*$ for an orthonormal basis $Q$ of $\ran(A^*)$, which can be computed from a
singular value decomposition (SVD) of $A$; this is the approach used in Section~\ref{sec:application}.
For large-scale problems, $P_{\ker(A)}z$ is the limit of the Landweber iteration
$x_{k+1}=x_k-\|A\|^{-2}A^*Ax_k$ with $x_0=z$
\cite[Section~6.1]{engl1996regularization}; the conjugate gradient method applied to
$A^*Ax=0$ with $x_0=z$ can be used as well \cite{kammerer1972convergence}.
\end{remark}

Null-space networks preserve data consistency but leave the component of the initial reconstruction in
$\ker(A)^\perp$ unchanged. Learned corrections in this component may improve reconstruction quality by enhancing
image structures and suppressing noise. In what follows, we extend null-space networks to allow such corrections
under an explicit data-space correction budget, thereby preserving quantitative control of the data residual.

\section{Data-proximal null-space networks}
\label{sec:DP}

Throughout this section, $((\Bo_\alpha)_{\alpha>0},\alpha^\star)$
is a possibly nonlinear regularization method for \eqref{eq:ip} based on a least-squares right inverse $B$ in the sense of
Definition~\ref{def:reg}. Convergence and data-proximity rates are understood as in Definition~\ref{def:rates}.
We introduce data-proximal null-space networks and establish convergence, convergence rates, and inheritance of
data-proximity rates under suitable assumptions and parameter choices.

\subsection{Data-proximal networks}

Whereas null-space networks are precisely the maps with $Af=A$
(Section~\ref{sec:nullspace}), data-proximal null-space networks are maps that change the
data by at most $\beta$, that is, $\|A(f(z)-z)\|_{\Y}\leq\beta$ for all $z\in\X$. The following definition
parametrizes such maps.

\begin{definition}[Data-proximal null-space networks]
\label{def:data-prox}
Let $(U_\theta,V_\theta)_{\theta\in\Theta}$ be a family of Lipschitz
continuous mappings $U_\theta,V_\theta\colon\X\to\X$.
For every $\beta>0$, let
$\Phi_\beta\colon\Y\to\Y$ satisfy
\begin{enumerate}[label=(\roman*)]
\item\label{item:phi-bound} $\forall\,w\in\Y\colon\ \|\Phi_\beta(w)\|_{\Y}\leq\beta$;
\item $\Phi_\beta(\ran(A))\subseteq\ran(A)$;
\item $\Phi_\beta(0)=0$.
\end{enumerate}
\noindent For $\theta\in\Theta$ and $\beta>0$, the mappings $f_{\theta,\beta}\colon\X\to\X$ defined by
\begin{equation}\label{eq:data-prox}
 f_{\theta,\beta}(z)
 \coloneqq
 z+P_{\ker(A)}U_\theta(z)
 +A^\mpi\Phi_\beta\bigl(AV_\theta(z)\bigr)
\end{equation}
are called data-proximal null-space networks for $A$. The parameter $\beta$ is called the
data-space correction budget, or briefly the budget.
\end{definition}

The third term in \eqref{eq:data-prox}, $A^\mpi\Phi_\beta(AV_\theta(z))$, is a correction in $\ker(A)^\perp$ 
and is well-defined even when $A^\mpi$ is unbounded, since $\Phi_\beta$ maps $\ran(A)$ into $\ran(A)$.
For $V_\theta=0$, the architecture reduces, by condition (iii), to the null-space network $\id+P_{\ker(A)}U_\theta$.
In general, exact data consistency is relaxed to a bound $\beta$ on the data change.
The name reflects the construction, which
decomposes the correction into a null-space component, as for null-space networks, and a component
in $\ker(A)^\perp$ whose effect on the data is controlled by $\beta$, so that the correction is no longer
confined to $\ker(A)$. Null-space networks are
contained as the special case $V_\theta=0$ and, for the scalar realization with fixed $\theta$,
arise as the limit $\beta\to0$ (Corollary~\ref{cor:fixed-theta}).%

The following lemma shows that $\beta$ bounds the data change introduced by the network and, consequently,
the possible increase in the initial data residual.

\begin{lemma}[Data proximity]
\label{lem:dp-proximity}
For all $z\in\X$, $\theta\in\Theta$, and $\beta>0$, we have
$\|A(f_{\theta,\beta}(z)-z)\|_{\Y}\leq\beta$. Consequently,
$\|Af_{\theta,\beta}(z)-\ydata^\delta\|_{\Y}\leq\|Az-\ydata^\delta\|_{\Y}+\beta$
for every $\ydata^\delta\in\Y$.
\end{lemma}

\begin{proof}
Since $AP_{\ker(A)}=0$ and $AA^\mpi w=w$ for $w\in\ran(A)$, we have
$A(f_{\theta,\beta}(z)-z)=\Phi_\beta(AV_\theta(z))$. Condition~\ref{item:phi-bound} in
Definition~\ref{def:data-prox} gives the first bound, and the triangle inequality the second.
\end{proof}

In particular, an initial residual bounded by $\delta$ yields a post-processed residual bounded by $\delta+\beta$.
This bound holds uniformly in $\theta$, since the budget $\beta$ is prescribed separately from
the trainable network parameters.
The following example realizes $\Phi_\beta$ by scalar rescaling and gives explicit formulas for architectures
with separate or shared networks.

\begin{example}[Scalar data-space corrections]
\label{ex:scalar-dp}
Let $\phi_\beta\colon[0,\infty)\to[0,\infty)$ satisfy
$t\phi_\beta(t)\leq\beta$ for all $t\geq0$, and set
$\Phi_\beta(w)=\phi_\beta(\|w\|_{\Y})w$.
Using $A^\mpi A=P_{\ker(A)^\perp}$,
\eqref{eq:data-prox} takes the form
\begin{equation}\label{eq:data-prox-scalar}
 f_{\theta,\beta}(z)
 =z+P_{\ker(A)}U_\theta(z)
 +\phi_\beta(\|AV_\theta(z)\|_{\Y})
  P_{\ker(A)^\perp}V_\theta(z).
\end{equation}
Here, $\theta$ collects the parameters of both networks, which may be separate or shared.
In the shared case $V_\theta=U_\theta$, a single network determines both components, and
\eqref{eq:data-prox-scalar} holds with $V_\theta$ replaced by $U_\theta$. In Section~\ref{sec:application},
the shared and separate realizations are denoted by DP1 and DP2, respectively.
Both realizations avoid a separate application of
$A^\mpi$ once the orthogonal projections are available.
The main instance is radial clipping, given by $\phi_\beta(t)=\min\{1,\beta/t\}$ for $t>0$ and
$\phi_\beta(0)=1$, that is, $\Phi_\beta(w)=w$ for $\|w\|_{\Y}\leq\beta$ and
$\Phi_\beta(w)=\beta w/\|w\|_{\Y}$ otherwise. It satisfies conditions (i)--(iii) in
Definition~\ref{def:data-prox} and acts as the identity on $\{w:\|w\|_{\Y}\leq\beta\}$.
\end{example}

Conversely, the architecture in Definition~\ref{def:data-prox} does not restrict the class of
such maps.

\begin{lemma}[Characterization]
\label{lem:dp-characterization}
Let $\beta>0$, and let $f\colon\X\to\X$ be Lipschitz continuous with
$\|A(f(z)-z)\|_{\Y}\leq\beta$ for all $z\in\X$. Then $f=f_{\theta,\beta}$ with
$U_\theta=V_\theta=f-\id$ and the radial clipping $\Phi_\beta$ of Example~\ref{ex:scalar-dp}.
\end{lemma}

\begin{proof}
Let $h\coloneqq f-\id$, which is Lipschitz continuous. Since $\|Ah(z)\|_{\Y}\leq\beta$ for all $z$,
we have $\Phi_\beta\circ Ah=Ah$, and $A^\mpi A=P_{\ker(A)^\perp}$ yields
$f_{\theta,\beta}=\id+P_{\ker(A)}h+P_{\ker(A)^\perp}h=f$.
\end{proof}

Together with Lemma~\ref{lem:dp-proximity}, the Lipschitz continuous data-proximal null-space
networks are thus exactly the Lipschitz continuous maps that change the data by at
most $\beta$, in analogy
to the characterization of null-space networks in Section~\ref{sec:nullspace}.

\subsection{Data-proximal post-processing}
\label{sec:dp-post}
We now combine data-proximal null-space networks with an initial regularization method.

\begin{definition}[Data-proximal post-processing]
\label{def:dp-post}
For a regularization parameter $\alpha>0$, a budget $\beta>0$, and network parameters
$\theta\in\Theta$, write $\gamma=(\alpha,\beta,\theta)\in\Gamma\coloneqq(0,\infty)\times(0,\infty)\times\Theta$.
We call $\Ro_\gamma\coloneqq f_{\theta,\beta}\circ\Bo_\alpha$ the data-proximal post-processing of the
initial regularization. A parameter choice rule is a map
$\gamma^\star\colon(0,\infty)\times\Y\to\Gamma$, $(\delta,\ydata^\delta)\mapsto(\alpha^\star,\beta^\star,\theta^\star)$,
and Definition~\ref{def:reg} applies verbatim with $(0,\infty)$ replaced by $\Gamma$, since it
only requires continuity of each $\Ro_\gamma$ and a parameter choice rule, but no convergence of
the parameters.
\end{definition}

The construction contains two basic special cases. For $U_\theta=V_\theta=0$, we obtain
$\Ro_\gamma=\Bo_\alpha$, that is, the initial regularization itself, and no network is trained; for
variational regularization, see Example~\ref{ex:var-reg}. For $U_\theta=V_\theta$ and the formal choice
$\Phi=\id_{\Y}$, which violates condition~\ref{item:phi-bound} in Definition~\ref{def:data-prox} for every
finite $\beta$, we obtain the standard residual network $f_\theta=\id+U_\theta$ and the residual
post-processing $\Ro_\gamma=(\id+U_\theta)\circ\Bo_\alpha$ of \cite{jin2017deep}; for $A^\mpi$
as initial reconstruction, this gives $(\id+U_\theta)\circ A^\mpi$ on $D(A^\mpi)$. Without a data-space
bound, such networks do not in general satisfy the bound of Lemma~\ref{lem:dp-proximity} for a prescribed
finite $\beta$. Further architectures from the literature are related to the construction as follows.

\begin{remark}[Related architectures]
\label{rem:related-architectures}
The null-space post-processing $(\id+P_{\ker(A)}U)\circ\Bo_\alpha$ of Definition~\ref{def:nsn}, which for
$B=A^\mpi$ gives the regularized null-space networks of \cite{schwab2019deepnullspacelearning}, is the special case
$V_\theta=0$ and $U_\theta=U$ of Definition~\ref{def:dp-post}; it preserves the initial data residual exactly but
leaves the component in $\ker(A)^\perp$ unchanged. The range-nullspace decomposition
$(\id+P_{\ker(A)}U_\theta+P_{\ker(A)^\perp}V_\theta)\circ A^\mpi$ of \cite{chen2020deep} and the regularizing networks
$(\id+U_{\theta(\alpha)})\circ\Bo_\alpha$ of \cite{Schwab2020}, with a parameter rule $\theta=\theta(\alpha)$, have the
same form with the formal choice $\Phi=\id_{\Y}$, which corresponds to Definition~\ref{def:dp-post} with an unbounded
budget; in general, they do not satisfy the bound of Lemma~\ref{lem:dp-proximity} for any finite $\beta$. Compared
with \cite{chen2020deep}, our construction replaces $A^\mpi$, which is unstable in the ill-posed setting, by a possibly
nonlinear regularization and bounds the data change of the correction in $\ker(A)^\perp$. Compared with
\cite{Schwab2020}, it treats the two components separately, bounds the data change, and chooses $\theta$ and $\beta$
independently, and our analysis replaces uniform Lipschitz bounds by calmness at the exact solutions.
\end{remark}
 
The parameter $\beta$ controls the data change and is not intended to be optimized
during network training. The theoretical bounds assume exact forward operators
and projections. Numerical approximations may introduce additional errors and therefore require
separate control of the data change.

\begin{remark}[Approximate projections]
\label{rem:numerical-implementation}
In practice, $P_{\ker(A)}$ can only be computed approximately, for instance as in
Remark~\ref{rem:projection-layer}. If $P_{\ker(A)}$ is replaced by an approximation $\Pi$ and
$P_{\ker(A)^\perp}$ by $\id-\Pi$ in \eqref{eq:data-prox-scalar}, while the factor
$\phi_\beta(\|AV_\theta(z)\|_{\Y})$ of the radial clipping is retained, the resulting network
$f_{\theta,\beta}^\Pi$ satisfies
$\|A(f_{\theta,\beta}^\Pi(z)-z)\|_{\Y}\leq \beta+\|A\Pi U_\theta(z)\|_{\Y}+\|A\Pi V_\theta(z)\|_{\Y}$.
Hence, the data-proximity bound is preserved up to the errors of the approximate null-space projection.
In Section~\ref{sec:application}, the matrix $\matid_n-\mathbf Q_p\mathbf Q_p^\top$ plays the role of $\Pi$, and
the corresponding term $\|A\Pi U_\theta(z)\|_{\Y}$ is the leakage. If the clipping factor is instead computed
from the approximate component, $\phi_\beta(\|A(\id-\Pi)V_\theta(z)\|_{\Y})$, the bound improves to
$\beta+\|A\Pi U_\theta(z)\|_{\Y}$; this is the case for DP2 in Section~\ref{sec:application}.
\end{remark}

We next establish conditions on the choice of the network parameters and the budget that ensure convergence to the LS right inverse $(\id+P_{\ker(A)}U)\circ B$ 
and inheritance of the data-proximity rate $r$ from the initial regularization. The limiting solution selection 
may therefore differ from that of the initial method.

\subsection{Convergence analysis}
\label{sec:convergence}

Let $\M\subseteq B(\ran(A))$ be an admissible class of exact solutions, so that $B(Az)=z$ for $z\in\M$.
The following analysis allows nonlinear initial regularization operators $\Bo_\alpha$. Parameter choice rules 
written without arguments are evaluated at $(\delta,\ydata^\delta)$.

To control the effect of perturbations in the initial  reconstruction, we use calmness at the exact solutions.
For $z\in\X$ and $\rho>0$, define

\begin{equation}\label{eq:calm-const}
 c_{\theta,\beta}(z;\rho)
 \coloneqq
 \sup\biggl\{
 \frac{\|f_{\theta,\beta}(x)-f_{\theta,\beta}(z)\|_{\X}}
      {\|x-z\|_{\X}}
 \;:\; 0<\|x-z\|_{\X}\leq\rho
 \biggr\}\in[0,\infty].
\end{equation}
This is the calmness constant of $f_{\theta,\beta}$ at $z$ relative to the closed ball of radius $\rho$ centered at $z$;
see \cite{dontchev2014implicit}. It is bounded by the Lipschitz constant on that ball but may be substantially 
smaller, since it compares function values only with the value at $z$.

\begin{theorem}[Convergence]
\label{thm:convergence}
Assume that $f_{\theta,\beta}\colon\X\to\X$ is continuous for every $\theta\in\Theta$ and $\beta>0$.
Suppose there exist a Lipschitz continuous mapping $U\colon\X\to\X$ and parameter choice rules
\(\beta^\star\colon(0,\infty)\times\Y\to(0,\infty)\) and \(\theta^\star\colon(0,\infty)\times\Y\to\Theta\)
such that
\begin{equation}\label{eq:conv-cond}
 \forall\,z\in\M:\quad
 \lim_{\delta\to0}\sup\bigl\{
 \|f_{\theta^\star,\beta^\star}(z)-(\id+P_{\ker(A)}U)(z)\|_{\X}
 \;:\;\|\ydata^\delta-Az\|_{\Y}\leq\delta\bigr\}=0,
\end{equation}
and, for every $z\in\M$, there exists $\rho_z>0$ such that
\begin{equation}\label{eq:calm-cond}
 \lim_{\delta\to0}\sup\bigl\{
 c_{\theta^\star,\beta^\star}(z;\rho_z)\,\|\Bo_{\alpha^\star}(\ydata^\delta)-z\|_{\X}
 \;:\;\|\ydata^\delta-Az\|_{\Y}\leq\delta\bigr\}=0,
\end{equation}
where $\theta^\star$, $\beta^\star$, and $\alpha^\star$
are evaluated at $(\delta,\ydata^\delta)$ and
$0\cdot\infty\coloneqq0$.
Then, with
$\gamma^\star\coloneqq
(\alpha^\star,\beta^\star,\theta^\star)$,
the following statements hold:
\begin{enumerate}[label=(\arabic*),ref=(\arabic*)]
\item\label{item:dp-convergence}
The pair $((\Ro_\gamma)_\gamma,\gamma^\star)$ defined
in Definition~\ref{def:dp-post} is a convergent regularization method, based on the LS right inverse
$(\id+P_{\ker(A)}U)\circ B$,
over
$(\id+P_{\ker(A)}U)(\M)$.

\item\label{item:dp-proximity}
If $\beta^\star=\mathcal O(\delta^r)$ uniformly for all $z\in\M$ and admissible noisy data
$\|\ydata^\delta-Az\|_{\Y}\leq\delta$, then $((\Ro_\gamma)_\gamma,\gamma^\star)$ is rate-$r$
data-proximal over $(\id+P_{\ker(A)}U)(\M)$, provided $((\Bo_\alpha)_\alpha,\alpha^\star)$
is rate-$r$ data-proximal over $\M$.
\end{enumerate}
\end{theorem}

\begin{proof}
For every fixed $\gamma=(\alpha,\beta,\theta)$, $\Ro_\gamma=f_{\theta,\beta}\circ\Bo_\alpha$ is continuous,
since both factors are continuous.
Let $\signal^\star=(\id+P_{\ker(A)}U)(z)$ with $z\in\M$ and $\|\ydata^\delta-Az\|_{\Y}\leq\delta$.
Since $B(Az)=z$ and the initial regularization converges, we have
$\|\Bo_{\alpha^\star}(\ydata^\delta)-z\|_{\X}\leq \rho_z$ for all sufficiently
small $\delta$, uniformly over the admissible noisy data. By the definition
\eqref{eq:calm-const} of the calmness constant,
\begin{equation*}
 \|\signal^\star-\Ro_{\gamma^\star}(\ydata^\delta)\|_{\X}
 \leq
 c_{\theta^\star,\beta^\star}(z;\rho_z)\,
 \|\Bo_{\alpha^\star}(\ydata^\delta)-z\|_{\X}
 +\|f_{\theta^\star,\beta^\star}(z)-(\id+P_{\ker(A)}U)(z)\|_{\X}.
\end{equation*}
By \eqref{eq:calm-cond} and \eqref{eq:conv-cond}, both terms converge to zero
uniformly over the admissible noisy data, which proves
assertion~\ref{item:dp-convergence}.
For assertion~\ref{item:dp-proximity}, note that $A\signal^\star=Az$. By Lemma~\ref{lem:dp-proximity}, we obtain
$\|\ydata^\delta- A\Ro_{\gamma^\star}(\ydata^\delta)\|_{\Y} \leq \|\ydata^\delta- A\Bo_{\alpha^\star}(\ydata^\delta)\|_{\Y} +\beta^\star \leq (\tau+C_\beta)\delta^r$,
where $\tau$ is the data-proximity constant of the initial regularization and $C_\beta>0$
is the uniform constant in $\beta^\star=\mathcal O(\delta^r)$. This proves data-proximity rate $r$ over
$(\id+P_{\ker(A)}U)(\M)$.
\end{proof}

Theorem~\ref{thm:convergence} extends the regularization results for null-space networks
\cite{schwab2019deepnullspacelearning} and regularizing networks \cite{Schwab2020} in two respects:
the initial regularization need not be linear, and uniform Lipschitz bounds on the networks are
replaced by the calmness condition \eqref{eq:calm-cond} at the exact solutions.
The same proof yields convergence for all $y\in \YLS$, including data with a component in
$\ran(A)^\perp$, if the initial regularization converges on $\YLS$ and \eqref{eq:conv-cond},
\eqref{eq:calm-cond} hold with $z$ and $Az$ replaced by $B(y)$ and $y$.

\begin{remark}[Sufficient conditions for \eqref{eq:calm-cond}]
\label{rem:calm-sufficient}
If the calmness constants are bounded along the parameter choice,
i.e., $\sup c_{\theta^\star,\beta^\star}(z;\rho_z)<\infty$ over all
sufficiently small $\delta$ and admissible noisy data, then
\eqref{eq:calm-cond} follows from the convergence of the initial
regularization. This is the case, in particular, if the family
$(f_{\theta,\beta})_{\theta,\beta}$ is uniformly Lipschitz continuous on
bounded sets. Such uniform Lipschitz continuity is, however, too strong
in the ill-posed setting of main interest. If $\ran(A)$ is not closed,
it fails as $\beta\to0$ already for the simplest scalar realization
(Example~\ref{ex:no-uniform-lip}). In contrast, for the scalar
realization with radial clipping and fixed $\theta$, the calmness
constant at any $z$ is bounded independently of $\beta$
(Lemma~\ref{lem:calm-scalar}).
\end{remark}

\begin{remark}[Data proximity and continuity]
The budget does not ensure the continuity of $f_{\theta,\beta}$ assumed in Theorem~\ref{thm:convergence},
nor does the bound $\|\Phi_\beta(w)\|_{\Y}\leq\beta$ imply \eqref{eq:conv-cond} or \eqref{eq:calm-cond}. If $\ran(A)$
is not closed,
a correction of norm at most $\beta$ in the data may correspond to an arbitrarily large correction 
$A^\mpi\Phi_\beta(w)$ in the image, so that $\beta^\star\to0$ does not by itself ensure convergence. The scalar
realization \eqref{eq:data-prox-scalar} with radial clipping avoids this effect, since
$A^\mpi\Phi_\beta(AV_\theta(z))=\phi_\beta(\|AV_\theta(z)\|_{\Y})\,P_{\ker(A)^\perp}V_\theta(z)$
is a rescaled network output, bounded by $\|P_{\ker(A)^\perp}V_\theta(z)\|_{\X}$
and vanishing as $\beta\to0$ for every fixed $z$
(Lemma~\ref{lem:calm-scalar}).
\end{remark}

Under additional approximation assumptions, we obtain convergence
rates. We distinguish the data-proximity rate $r$ from the
convergence rate $s$ and the network approximation rate $q$.

\begin{theorem}[Convergence rates]
\label{thm:dp-rates}
Under the assumptions of Theorem~\ref{thm:convergence}, let $\X_s\subseteq\M$ be bounded and suppose that
$((\Bo_\alpha)_{\alpha>0},\alpha^\star)$ is rate-$r$ data-proximal and has convergence rate $s\in(0,1]$ over $\X_s$ (Definition~\ref{def:rate}); by Remark~\ref{rem:rates-X-Y}, the
first condition holds automatically over $\X_s$ if $r\leq s$. Suppose further that the parameter choice rules $\theta^\star$ and $\beta^\star$ satisfy, for some $q>0$,
\begin{equation}\label{eq:rate-assumptions}
\sup\bigl\{\|f_{\theta^\star,\beta^\star}(z)-(\id+P_{\ker(A)}U)(z)\|_{\X}
\;:\;z\in\X_s,\ \|\ydata^\delta-Az\|_{\Y}\leq\delta\bigr\}=\mathcal O(\delta^q)
\quad\text{as }\delta\to0,
\end{equation}
and that $\beta^\star=\mathcal O(\delta^r)$ uniformly for all $z\in\X_s$ and admissible
noisy data $\ydata^\delta$.
Suppose, moreover, that the calmness constants are uniformly bounded on
$\X_s$, i.e., there exist $\rho_s>0$ and $C_L>0$ such that
$c_{\theta^\star,\beta^\star}(z;\rho_s)\leq C_L$ for all $z\in\X_s$, all
sufficiently small $\delta$, and all admissible noisy data.
Then, for $\signal^\star\in (\id+P_{\ker(A)}U)(\X_s)$ with $\|\ydata^\delta-A\signal^\star\|_{\Y}\leq\delta$,
we have
\begin{align}
	\|\signal^\star-\Ro_{\gamma^\star}(\ydata^\delta)\|_{\X}&=\mathcal O(\delta^{\min\{s,q\}}),\label{eq:rate1}\\
	\|\ydata^\delta-A\Ro_{\gamma^\star}(\ydata^\delta)\|_{\Y}&=\mathcal O(\delta^r).\label{eq:rate2}
\end{align}
That is, the regularization method $((\Ro_\gamma)_\gamma,\gamma^\star)$ is
rate-$r$ data-proximal and has convergence rate $\min\{s,q\}$ over $(\id+P_{\ker(A)}U)(\X_s)$.
In particular, if $q\geq s$, the convergence rate $s$ of the initial regularization is preserved.

\end{theorem}

\begin{proof}
Let $\signal^\star=(\id+P_{\ker(A)}U)(z)$ with $z\in\X_s$ and $\|\ydata^\delta-A\signal^\star\|_{\Y}\leq\delta$.
Since $A\signal^\star=Az$, the initial regularization has convergence rate $s$, and hence
$\|\Bo_{\alpha^\star}(\ydata^\delta)-z\|_{\X} \leq C_B\delta^s$
for some constant $C_B>0$ independent of $z\in\X_s$ and the admissible noisy data.
For sufficiently small $\delta$, we have $C_B\delta^s\leq \rho_s$, and hence, by
the uniform bound on the calmness constants,
\begin{align*}
\|\signal^\star-\Ro_{\gamma^\star}(\ydata^\delta)\|_{\X} &= \|(\id+P_{\ker(A)}U)(z) -f_{\theta^\star,\beta^\star} (\Bo_{\alpha^\star}(\ydata^\delta))\|_{\X}\\
&\leq \|(\id+P_{\ker(A)}U)(z) -f_{\theta^\star,\beta^\star}(z)\|_{\X} + \|f_{\theta^\star,\beta^\star}(z) -f_{\theta^\star,\beta^\star} (\Bo_{\alpha^\star}(\ydata^\delta))\|_{\X}\\
&\leq \|(\id+P_{\ker(A)}U)(z) -f_{\theta^\star,\beta^\star}(z)\|_{\X}+ C_L\|z-\Bo_{\alpha^\star}(\ydata^\delta)\|_{\X}\\
&\leq C_q\delta^q+C_LC_B\delta^s =\mathcal O(\delta^{\min\{s,q\}}),
\end{align*}
where $C_q>0$ is the uniform constant in
\eqref{eq:rate-assumptions}.
This proves \eqref{eq:rate1}.

For the data-proximity estimate \eqref{eq:rate2}, we use Lemma~\ref{lem:dp-proximity} and obtain
\begin{align*}
\|\ydata^\delta-A\Ro_{\gamma^\star}(\ydata^\delta)\|_{\Y} &= \|\ydata^\delta-Af_{\theta^\star,\beta^\star}(\Bo_{\alpha^\star}(\ydata^\delta))\|_{\Y}\\
&\leq \|\ydata^\delta-A\Bo_{\alpha^\star}(\ydata^\delta)\|_{\Y}+\beta^\star\\
&\leq \tau\delta^r+C_\beta\delta^r=(\tau+C_\beta)\delta^r =\mathcal O(\delta^r),
\end{align*}
where $\tau>0$ is the data-proximity constant of the initial regularization and $C_\beta>0$ is the
uniform constant in $\beta^\star=\mathcal O(\delta^r)$. This proves \eqref{eq:rate2}.
\end{proof}

For pointwise rather than uniform convergence rates,
the boundedness assumption on $\X_s$ can be dropped, with constants depending on the individual solution.

\subsection{Fixed networks with radial clipping}
\label{sec:fixed}
\label{subsec:practical aspects}

We now specialize the construction to the radial clipping of Example~\ref{ex:scalar-dp}. This choice yields convergence guarantees for fixed networks
and a projection-based implementation, while illustrating why calmness is more suitable than uniform Lipschitz bounds
in ill-posed settings.
For fixed $\theta$ and $z$, clipping is inactive whenever $\beta\geq\|AV_\theta(z)\|_{\Y}$, and the network reduces to
\(z+P_{\ker(A)}U_\theta(z) +P_{\ker(A)^\perp}V_\theta(z).\) For $V_\theta=U_\theta$, this is the standard residual
network $\id+U_\theta$ evaluated at $z$.

Radial clipping is the metric projection onto the closed ball of radius
$\beta$ in $\Y$ and hence nonexpansive. If $\ran(A)$ is closed and the
Lipschitz constants of $U_\theta$ and $V_\theta$ are uniformly bounded by
$L_U$ and $L_V$, respectively, then
$\|f_{\theta,\beta}(x)-f_{\theta,\beta}(z)\|_{\X} \leq(1+L_U+\|A^\mpi\|\,\|A\|L_V)\|x-z\|_{\X}$
for all $x,z\in\X$, uniformly in $\theta$ and $\beta$.  If $\ran(A)$ is not closed, such a uniform bound need
not exist, even for the simple architecture considered in the following example.

\begin{example}[Failure of uniform Lipschitz continuity]
\label{ex:no-uniform-lip}
Let $\ran(A)$ be non-closed, $U_\theta=0$, $V_\theta=\id$, and let
$\Phi_\beta$ be radial clipping, so that
$f_\beta(v)=v+\phi_\beta(\|Av\|_{\Y})v$ for $v\in \ker(A)^\perp$.
Let $R>0$ and $0<\beta\leq R\|A\|/10$. Since $A$ is not bounded below on
$\ker(A)^\perp$, there is $w\in \ker(A)^\perp$ with $\|w\|_{\X}=1$ and
$\|Aw\|_{\Y}\leq\beta/R$. Moreover, there is $e\in \ker(A)^\perp$ with
$\|e\|_{\X}=1$ and $\|Ae\|_{\Y}\geq\|A\|/2$. Set $x=(R/2)\,w$ and
$z=x+\varepsilon e$ with $\varepsilon=5\beta/\|A\|\leq R/2$. Then
$\|Ax\|_{\Y}\leq\beta/2$ and
$\|Az\|_{\Y}\geq\varepsilon\|A\|/2-\beta/2=2\beta$, so that
$\phi_\beta(\|Ax\|_{\Y})=1$ and $\phi_z\coloneqq\phi_\beta(\|Az\|_{\Y})\leq1/2$.
Consequently,
$\|f_\beta(z)-f_\beta(x)\|_{\X} =\|(1+\phi_z)(z-x)-(1-\phi_z)x\|_{\X} \geq R/4-2\varepsilon$.
Since $x,z\in\overline{B_R(0)}$ and $\|z-x\|_{\X}=\varepsilon$, the Lipschitz
constant of $f_\beta$ on $\overline{B_R(0)}$ is at least
$\|A\|R/(20\beta)-2$, which tends to infinity as $\beta\to0$.
\end{example}

The example rules out uniform Lipschitz bounds on bounded sets in general. For convergence, 
however, it suffices to control perturbations relative to each exact solution. For fixed network 
parameters, radial clipping provides such a bound uniformly in $\beta$.

For $z\in\X$, set $h_\theta(z)\coloneqq P_{\ker(A)^\perp}V_\theta(z)$ and define
\[
\kappa_\theta(z)\coloneqq
\begin{cases}
\|h_\theta(z)\|_\X/\|Ah_\theta(z)\|_\Y,
  & h_\theta(z)\neq0,\\
0, & h_\theta(z)=0.
\end{cases}
\]
This ratio measures the size of the proposed correction relative to its effect on the data. The following lemma
uses it to bound calmness uniformly in $\beta$ and to quantify the deviation from the associated null-space network.

\begin{lemma}[Calmness of the scalar realization]
\label{lem:calm-scalar}
Fix $\theta$ and let $f_{\theta,\beta}$ be given by
\eqref{eq:data-prox-scalar}, with
$\phi_\beta(t)=\min\{1,\beta/t\}$ for $t>0$
and $\phi_\beta(0)=1$.
Let $U_\theta,V_\theta$ be Lipschitz continuous with
constants $L_{U,\theta},L_{V,\theta}$, respectively.
For any $z\in\X$, set $h=h_\theta(z)$.
Then, for all $x\in\X$ and $\beta>0$,
\begin{align}
&\begin{aligned}
&\|f_{\theta,\beta}(x)-f_{\theta,\beta}(z)\|_{\X}\leq\Bigl(1+L_{U,\theta}
+L_{V,\theta}\bigl(1+\|A\|\kappa_\theta(z)\bigr)\Bigr)\|x-z\|_{\X}.
\end{aligned}
\label{eq:calm-scalar}\\
&\|f_{\theta,\beta}(z)
  -(\id+P_{\ker(A)}U_\theta)(z)\|_{\X}
=\min\{\beta\kappa_\theta(z),\|h\|_{\X}\}.
\label{eq:approx-scalar}
\end{align}
Here $\kappa_\theta(z)$ is finite, since $Ah\neq0$ for every nonzero $h\in\ker(A)^\perp$.
Consequently, $f_{\theta,\beta}$ is continuous on $\X$
for every fixed $\theta$ and $\beta>0$.
If $\ran(A)$ is closed, then additionally
$\kappa_\theta(z)\leq\|A^\mpi\|$ for all $z\in\X$.
\end{lemma}

\begin{proof}
Fix $\beta>0$.
We first show that, for all $s,t\geq0$,
\begin{equation}\label{eq:phi-rel-lip}
	|\phi_\beta(s)-\phi_\beta(t)| \leq\frac{|s-t|}{\max\{\beta,t\}}.
\end{equation}
If $s,t\leq\beta$, the left-hand side vanishes. For the remaining cases,
\begin{align*}
&|\beta/s-\beta/t| =\beta|s-t|/(st) \leq|s-t|/t,
&& s,t\geq\beta,\\
&1-\beta/t =(t-\beta)/t \leq(t-s)/t,
&& s<\beta\leq t,\\
&1-\beta/s =(s-\beta)/s \leq(s-t)/\beta,
&& t<\beta\leq s.
\end{align*}
These inequalities establish \eqref{eq:phi-rel-lip}.

Let $g\coloneqq P_{\ker(A)^\perp}V_\theta(x)$, $s=\|Ag\|_{\Y}$, and $t=\|Ah\|_{\Y}$.
Since $AP_{\ker(A)^\perp}=A$, we obtain
$|s-t| \leq\|A(V_\theta(x)-V_\theta(z))\|_{\Y} \leq\|A\|L_{V,\theta}\|x-z\|_{\X}$ and $\|g-h\|_{\X} \leq L_{V,\theta}\|x-z\|_{\X}$.
Using $\phi_\beta\leq1$ and \eqref{eq:phi-rel-lip}, 
\begin{align*}
\|\phi_\beta(s)g-\phi_\beta(t)h\|_{\X} &\leq \phi_\beta(s)\|g-h\|_{\X} +|\phi_\beta(s)-\phi_\beta(t)|\,\|h\|_{\X}\\
&\leq L_{V,\theta} \Bigl(1+\|A\|\frac{\|h\|_{\X}}{\max\{\beta,t\}}\Bigr) \|x-z\|_{\X}.
\end{align*}
Since $\|h\|_{\X}/\max\{\beta,t\}\leq\kappa_\theta(z)$,
including when $h=0$, adding the bounds for the identity and null-space terms 
proves \eqref{eq:calm-scalar}. Its coefficient is finite for every fixed $z$ and $\theta$
and independent of $\beta$, which also proves continuity.

Moreover, \(f_{\theta,\beta}(z) -(\id+P_{\ker(A)}U_\theta)(z)=\phi_\beta(t)h\) and \(\phi_\beta(t)\|h\|_{\X} =\beta\|h\|_{\X}/\max\{\beta,t\}
=\min\{\beta\kappa_\theta(z),\|h\|_{\X}\}\) which proves \eqref{eq:approx-scalar}. Finally, if $\ran(A)$ is closed, then
\(\|h\|_{\X} =\|A^\mpi Ah\|_{\X} \leq\|A^\mpi\|\,\|Ah\|_{\Y},\) and hence $\kappa_\theta(z)\leq\|A^\mpi\|$,
with the case $h=0$ following from its definition.
\end{proof}

The following corollary treats networks trained once and then kept fixed. It establishes convergence as the budget
vanishes, without an additional coupling to the initial regularization parameter.

\begin{corollary}[Fixed network]
\label{cor:fixed-theta}
Let $\theta$ be fixed and let $U\coloneqq U_\theta$ and $V\coloneqq V_\theta$ be Lipschitz continuous with constants
$L_{U,\theta}$ and $L_{V,\theta}$, respectively. Let $f_\beta\coloneqq f_{\theta,\beta}$ be the scalar realization with radial
clipping. Let $((\Bo_\alpha)_{\alpha>0},\alpha^\star)$ be a convergent regularization method based on $B$ over $\M$,
and let $\beta^\star\colon(0,\infty)\times\Y\to(0,\infty)$ satisfy, for every $z\in\M$,
$\lim_{\delta\to0}\sup\{\beta^\star(\delta,\ydata^\delta):\|\ydata^\delta-Az\|_{\Y}\leq\delta\}=0$.
Define $\Ro_{\alpha,\beta}\coloneqq f_\beta\circ\Bo_\alpha$, that is, $\Ro_\gamma$ of Definition~\ref{def:dp-post}
with $\gamma=(\alpha,\beta,\theta)$ and fixed $\theta$. Then the following statements hold:
\begin{enumerate}[label=(\arabic*),ref=(\arabic*)]
\item\label{item:fixed-conv}
The family $(\Ro_{\alpha,\beta})_{\alpha,\beta>0}$, with parameter choice
$(\alpha^\star,\beta^\star)$, is a convergent regularization method, based on the LS right inverse
$(\id+P_{\ker(A)}U)\circ B$,
over $(\id+P_{\ker(A)}U)(\M)$. No additional coupling between $\beta^\star$ and
$\alpha^\star$ is required.

\item\label{item:fixed-rates}
Suppose, in addition, that the initial regularization is rate-$r$
data-proximal and has convergence rate $s$ over a bounded set
$\X_s\subseteq\M$, that $\beta^\star=\mathcal O(\delta^r)$ uniformly over
$z\in\X_s$ and admissible noisy data, and that 
$K_s\coloneqq\sup_{z\in\X_s}\kappa_\theta(z)<\infty$.
Then the post-processing method is rate-$r$ data-proximal and
has convergence rate $\min\{s,r\}$ over $(\id+P_{\ker(A)}U)(\X_s)$.

\item If $\ran(A)$ is closed, then $K_s\leq\|A^\mpi\|$ for every nonempty
$\X_s\subseteq\X$.
\end{enumerate}
\end{corollary}

\begin{proof}
By Lemma~\ref{lem:calm-scalar}, $f_\beta$ is continuous for every $\beta>0$.
Hence $\Ro_{\alpha,\beta}=f_\beta\circ\Bo_\alpha$ is continuous for every fixed $\alpha,\beta>0$.
By \eqref{eq:calm-scalar}, $c_{\theta,\beta}(z;\rho)\leq 1+L_{U,\theta}+L_{V,\theta}(1+\|A\|\kappa_\theta(z))$
for all $\rho,\beta>0$. This bound is independent of $\beta$, so \eqref{eq:calm-cond} holds by
Remark~\ref{rem:calm-sufficient}. By \eqref{eq:approx-scalar}, the supremum of the approximation error in
\eqref{eq:conv-cond} is bounded by $\kappa_\theta(z)\sup\beta^\star$ and tends to zero,
where the supremum is taken over $\|\ydata^\delta-Az\|_{\Y}\leq\delta$.
Theorem~\ref{thm:convergence}, with $\theta^\star\equiv\theta$, yields \ref{item:fixed-conv}.
Under the assumptions of \ref{item:fixed-rates},
the calmness constants are bounded by $1+L_{U,\theta}+L_{V,\theta}(1+\|A\|K_s)$ uniformly on $\X_s$, and
\eqref{eq:rate-assumptions} holds with $q=r$, since the approximation error is bounded by
$\beta^\star K_s=\mathcal O(\delta^r)$ uniformly over $\X_s$ and admissible noisy data.
Theorem~\ref{thm:dp-rates} yields \ref{item:fixed-rates}. The last assertion follows from
Lemma~\ref{lem:calm-scalar}.
\end{proof}

\begin{remark}[Interpretation]
\label{rem:fixed-theta}
The quantity $\kappa_\theta(z)$ measures how strongly the correction proposed by the network 
at $z$ points into nearly invisible directions of $\ker(A)^\perp$. It captures the sensitivity of this 
correction in the calmness and approximation bounds. The condition $K_s<\infty$ provides the uniform 
control over $\X_s$ required for convergence rates. Since it concerns only the corrections 
proposed by the network, it may hold even when $\ran(A)$ is not closed. Without $K_s<\infty$, the 
proof still yields pointwise rates, provided the corresponding bounds for the initial regularization
and the budget hold:
\[
\|\signal^\star-\Ro_{\gamma^\star}(\ydata^\delta)\|_{\X}
\leq\bigl(1+L_{U,\theta}+L_{V,\theta}(1+\|A\|\kappa_\theta(z))\bigr)C_B\delta^s
+C_\beta\kappa_\theta(z)\,\delta^r,
\]
where $\signal^\star=(\id+P_{\ker(A)}U)(z)$ and the constants
may depend on the individual solution.

By \eqref{eq:approx-scalar}, the correction in $\ker(A)^\perp$
vanishes at every fixed input as the budget tends to zero.
Under the assumptions of Corollary~\ref{cor:fixed-theta}, the limiting
solution selection coincides with that of the null-space network
$\id+P_{\ker(A)}U$. The additional flexibility thus lies at positive
noise levels. At any input $z$, the correction's data-space effect
is bounded by $\beta$, while its image-space magnitude is exactly
$\min\{\beta\kappa_\theta(z),\|P_{\ker(A)^\perp}V_\theta(z)\|_{\X}\}$.
For the post-processed reconstruction, this identity is evaluated at
$z=\Bo_{\alpha^\star}(\ydata^\delta)$ and $\beta=\beta^\star$.
\end{remark}

\section{Limited-angle tomography and microlocal consistency}
\label{sec:microlocal}

In this section, we consider limited-angle computed tomography (CT) modeled by the Radon transform 
and investigate the microlocal implications of data consistency and data proximity.
We show that exact data consistency implies microlocal consistency on the visible region, 
while the data norm is proportional to the homogeneous Sobolev norm of order $-1/2$ of the visible
component, so that data proximity controls the visible reconstruction error in this norm. The numerical implementation and experiments are presented
separately in Section~\ref{sec:application}.

Throughout this section, the spaces $X$ and $Y$ of Sections~\ref{sec:background} and~\ref{sec:DP} are
$X=L^2(\mathbb{R}^2)$ and $Y=L^2(\acqangles\times\mathbb{R})$, where $\acqangles\subseteq\mathbb{R}$ is a closed interval
of measured angles with $0<|\acqangles|<\pi$, and the limited-angle Radon transform $\mathcal R_{\acqangles}$ (to be defined in
Definition~\ref{def:radon-L2} below) takes the role of $A$. In contrast to Sections~\ref{sec:background} and~\ref{sec:DP}, $\mathcal R_{\acqangles}$ is a densely
defined, closed, but unbounded operator from $D(\mathcal R_{\acqangles})\subseteq X$ to $Y$. The arguments of
Section~\ref{sec:DP} extend to this setting under additional domain assumptions, as explained in
Remark~\ref{rem:radon-unbounded}. Since $X$ and $Y$ are now function spaces, we denote their elements by $u$ and $g$,
and reserve $x\in\mathbb{R}^2$ for the spatial variable. Integral transforms acting on functions are denoted by
calligraphic letters; in particular, $\mathcal F u(\xi)=\int_{\mathbb R^2}e^{-ix\cdot\xi}u(x)\,dx$, with
$\widehat u\coloneqq\mathcal F u$, denotes the two-dimensional Fourier transform and
$\mathcal F_s g(\varphi,\sigma)=\int_{\mathbb{R}}e^{-is\sigma}g(\varphi,s)\,ds$ the Fourier transform in the detector
variable.

\subsection{The Radon transform and filtered backprojection}
\label{sec:radon-model}

For a smooth, compactly supported function $u\colon\mathbb{R}^2\to\mathbb{R}$, 
the two-dimensional Radon transform is defined by the line integral
\begin{equation}\label{eq:radon-line}
	\mathcal Ru(\varphi,s) =\int_{\mathbb{R}} u(s\radondir(\varphi)+t\radondir^\perp(\varphi))\,dt,
\end{equation}
where $\radondir(\varphi)=(\cos\varphi,\sin\varphi)$ and $\radondir^\perp(\varphi)=(-\sin\varphi,\cos\varphi)$.
Thus, $\mathcal Ru(\varphi,s)$ integrates $u$ over the line $\{x\in\mathbb{R}^2:x\cdot\radondir(\varphi)=s\}$.
Since $\mathcal Ru(\varphi+\pi,s)=\mathcal Ru(\varphi,-s)$, full-angle data are determined by any angular interval of
length $\pi$. In limited-angle tomography, only the angles in $\acqangles$ are measured, and the limited-angle Radon transform
$\mathcal R_{\acqangles}$ is the restriction of $\mathcal Ru$ to $\acqangles\times\mathbb{R}$.

For smooth, compactly supported $u$, the Fourier-slice theorem \cite{natterer2001mathematics,smith1977practical}
states that $\mathcal F_s(\mathcal R_{\acqangles} u)(\varphi,\sigma)=\widehat u(\sigma\radondir(\varphi))$ for all
$(\varphi,\sigma)\in\acqangles\times\mathbb{R}$. For general $u\in L^2(\mathbb{R}^2)$, the line integrals in
\eqref{eq:radon-line} need not exist, so \eqref{eq:radon-line} does not define $\mathcal R_{\acqangles}$ on all of $X$.
We take the Fourier-slice identity as the definition instead; this yields a closed operator whose domain, kernel and
Moore--Penrose inverse are explicit in the Fourier domain.

\begin{definition}[Limited-angle Radon transform on $L^2(\mathbb{R}^2)$]
\label{def:radon-L2}
Let $V_{\acqangles}\coloneqq\{\sigma\radondir(\varphi): \varphi\in\operatorname{int}(\acqangles),\ \sigma\neq0\}$
denote the double wedge of measured frequencies. The limited-angle Radon transform
$\mathcal R_{\acqangles}\colon D(\mathcal R_{\acqangles})\subseteq L^2(\mathbb{R}^2)\to L^2(\acqangles\times\mathbb{R})$ is defined
on
$D(\mathcal R_{\acqangles})\coloneqq\{u\in L^2(\mathbb{R}^2): \int_{V_{\acqangles}}|\widehat u(\xi)|^2\,|\xi|^{-1}\,d\xi<\infty\}$
by
\begin{equation}\label{eq:fourier-slice}
	\mathcal F_s(\mathcal R_{\acqangles} u)(\varphi,\sigma) =\widehat u(\sigma\radondir(\varphi))
	\quad\text{for almost every }(\varphi,\sigma)\in\acqangles\times\mathbb{R}.
\end{equation}
\end{definition}

Unless stated otherwise, $\mathcal R_{\acqangles}$ denotes this operator on $L^2(\mathbb{R}^2)$. By the Fourier-slice theorem, it agrees with \eqref{eq:radon-line}, restricted to
$\acqangles\times\mathbb{R}$, on smooth, compactly supported functions, as in the full-angle case
\cite{smith1977practical,goeppel2023radon}. By \eqref{eq:fourier-slice},
$\ker(\mathcal R_{\acqangles}) = \{u\in L^2(\mathbb{R}^2): \widehat u=0\text{ a.e. on }V_{\acqangles}\}$.

\begin{remark}[Filtered backprojection and the Moore--Penrose inverse]
\label{rem:fbp-mpi}
For smooth, compactly supported $u$ and full-angle data, the filtered backprojection (FBP) formula reads
$u=(4\pi)^{-1}\mathcal R^*\Lambda \mathcal Ru$ \cite{natterer2001mathematics,goeppel2023radon}. Here,
$\mathcal R^*g(x)=\int_0^{2\pi} g(\varphi,x\cdot\radondir(\varphi))\,d\varphi$ is the backprojection, and $\Lambda$ is the
Riesz potential in the detector variable, defined by $\mathcal F_s(\Lambda g)(\varphi,\sigma)=|\sigma|\,\mathcal F_sg(\varphi,\sigma)$.
For limited-angle data, the restricted backprojection
$\mathcal R_{\acqangles}^*g(x)=\int_{\acqangles}g(\varphi,x\cdot\radondir(\varphi))\,d\varphi$ yields, for
$u\in D(\mathcal R_{\acqangles})$,
\[
 \mathcal R_{\acqangles}^\mpi \mathcal R_{\acqangles}u
 =P_{\ker(\mathcal R_{\acqangles})^\perp}u
 =\mathcal F^{-1}(\chi_{V_{\acqangles}}\widehat u)
 =\frac{1}{2\pi}\mathcal R_{\acqangles}^*\Lambda \mathcal R_{\acqangles}u,
\]
where $\chi_{V_{\acqangles}}$ is the characteristic function of $V_{\acqangles}$ and the factor $1/(2\pi)$
accounts for integration over $\acqangles$ without antipodal duplication; cf.\
\cite[Theorem~2.1]{frikel2013artifacts} and \cite[Theorem~3.9]{frikel2013thesis}. For general
$u\in D(\mathcal R_{\acqangles})$, the right-hand side is understood through the Fourier multiplier. Thus,
limited-angle FBP with the exact ramp filter $\Lambda$ coincides with $\mathcal R_{\acqangles}^\mpi$ on $\ran(\mathcal R_{\acqangles})$.
For noisy data, this inversion is unstable, since $\mathcal R_{\acqangles}^\mpi$ is unbounded; in practice, $\Lambda$ is
therefore replaced by a smoothed ramp filter with a frequency cutoff, and the operators are discretized. Limited-angle FBP reconstructions typically exhibit streak artifacts, whose locations
and directions have been characterized microlocally in \cite{frikel2013artifacts,borg2018incompletect}.
\end{remark}

\begin{remark}[Relation to the bounded setting]
\label{rem:radon-unbounded}
Since $A=\mathcal R_{\acqangles}$ is unbounded \cite[Section~12]{smith1977practical}, \cite{goeppel2023radon}, it does not
directly fall within the bounded setting of Sections~\ref{sec:background} and~\ref{sec:DP}. The convergence and data-proximity arguments
of Theorems~\ref{thm:convergence} and~\ref{thm:dp-rates} carry over if the initial reconstructions and all
network outputs lie in $D(\mathcal R_{\acqangles})$ and the resulting maps are continuous, for instance if $V_\theta$ takes
values in functions supported in a fixed bounded set $K$, on which $\mathcal R_{\acqangles}$ is bounded
(Remark~\ref{rem:radon-compact-support}); bounds involving $\|A\|$, as in Lemma~\ref{lem:calm-scalar}, then use the
norm of $\mathcal R_{\acqangles}$ on $L^2(K)$. We work on $L^2(\mathbb{R}^2)$ because
there, as for the discretized operator in Section~\ref{sec:application}, $\mathcal R_{\acqangles}$ has a nontrivial null space,
which allows a direct comparison of null-space and data-proximal corrections; for compactly supported objects,
$\mathcal R_{\acqangles}$ is injective, and invisible features can only be reached by corrections that change the data
(Remark~\ref{rem:radon-compact-support}).
\end{remark}

\begin{remark}[Compactly supported functions]
\label{rem:radon-compact-support}
For functions supported in the unit disk $B_1(0)$, the restriction of $\mathcal R_{\acqangles}$ to $L^2(B_1(0))$ is a
bounded and injective operator into $L^2(\acqangles\times(-1,1))$
\cite{natterer2001mathematics}, so that nontrivial null-space corrections are not possible. More generally,
nonzero elements of $\ker(\mathcal R_{\acqangles})$ are never compactly supported, since their Fourier transform vanishes
on an open cone. For compactly supported objects, invisible information is not lost but is carried by components
with very small singular values \cite{davison1983illconditioned,louis1986incomplete}. Null-space networks
therefore cannot act on them, whereas data-proximal networks
can correct them while bounding the resulting change in the data. Moreover, the Moore--Penrose inverse of this restriction recovers
$u$ from exact data, whereas limited-angle FBP with the exact ramp filter reconstructs only the visible Fourier
components; the two therefore differ, in contrast to Remark~\ref{rem:fbp-mpi}.
\end{remark}

\subsection{Visible singularities and microlocal consistency}

Preserving visible singularities is an important principle in learning-based reconstruction
\cite{Bubba_2019,andrade2022deepmicrolocal}. To formalize microlocal consistency on the visible region,
we first recall the wavefront set \cite[Section~8.1]{hormander2003analysis}, which describes the locations
$x\in\mathbb{R}^2$ and directions $\xi\in\mathbb{R}^2\setminus\{0\}$ of the singularities of a function as a subset of
$\mathbb{R}^2\times(\mathbb{R}^2\setminus\{0\})$.

\begin{definition}[Wavefront set]
\label{def:wavefront}
Let $u\in L^2(\mathbb{R}^2)$. A point $(x,\xi)\in\mathbb{R}^2\times(\mathbb{R}^2\setminus\{0\})$ does not belong
to the wavefront set $\WF(u)$ if there exist $\chi\in C_c^\infty(\mathbb{R}^2)$ with $\chi(x)\neq0$ and an open
conic neighborhood $\Xi$ of $\xi$ such that $|\widehat{\chi u}(\eta)|\leq
C_N(1+|\eta|)^{-N}$ for all $\eta\in\Xi$ and $N\in\mathbb{N}$.
\end{definition}

Following \cite{quinto1993singularities,frikel2016generalized}, a singularity is called visible if it can be
stably recovered from the data; visibility is thus a property of the measurements, not of a particular reconstruction
method. For limited-angle Radon data, the visible singularities of $u$ are $\WF(u)\cap(\mathbb{R}^2\times V_{\acqangles})$, with the double wedge $V_{\acqangles}$ of
Definition~\ref{def:radon-L2}. Limited-angle FBP is known to preserve these singularities
\cite{borg2018incompletect,frikel2013artifacts,quinto1993singularities}. The directions $\sigma\radondir(\varphi)$ with
$\varphi\in\partial\acqangles$ are measured but do not belong to the open set $V_{\acqangles}$, so singularities in these
directions are not counted as visible; along these directions, limited-angle reconstructions may create streak
artifacts \cite{frikel2013artifacts,hahn2026approximate}.

\begin{definition}[Microlocal consistency]
\label{def:microlocal-consistency}
Let $V$ be an arbitrary subset of $\mathbb{R}^2\times(\mathbb{R}^2\setminus\{0\})$. 
Two functions $u_1,u_2\in L^2(\mathbb{R}^2)$ are called microlocally consistent on $V$ if $\WF(u_1-u_2)\cap V=\varnothing$.
In particular, this implies $\WF(u_1)\cap V=\WF(u_2)\cap V$.
\end{definition}

\subsection{Consequences of data consistency and data proximity}

The following proposition shows that, for the Radon transform, exact data consistency implies 
microlocal consistency, without requiring explicit wavefront-set extraction. Its proof uses only the
Fourier-slice theorem and the definition of the wavefront set.

\begin{proposition}[Data consistency implies microlocal consistency]
\label{prop:dc-microlocal}
If $u_1,u_2\in D(\mathcal R_{\acqangles})$ satisfy
$\mathcal R_{\acqangles}u_1=\mathcal R_{\acqangles} u_2$, then they are
microlocally consistent on $\mathbb{R}^2\times V_{\acqangles}$.
\end{proposition}

\begin{proof}
Let $h=u_1-u_2$. Then $\mathcal R_{\acqangles} h=0$, so by \eqref{eq:fourier-slice}, $\widehat h(\sigma\radondir(\varphi))=0$
for a.e.\ $(\varphi,\sigma)\in\acqangles\times\mathbb{R}$; thus, $\widehat h$ vanishes almost everywhere on the open double
wedge $V_{\acqangles}$ of Definition~\ref{def:radon-L2}.
Let $x_0\in\mathbb{R}^2$ and $\xi_0\in V_{\acqangles}$, and choose open cones $\Xi_1,\Xi_2$ such that
$\xi_0\in\Xi_1$ and $\overline{\Xi_1}\setminus\{0\}\mathrel{\subseteq}\Xi_2\mathrel{\subseteq} V_{\acqangles}$. Let $\chi\in
C_c^\infty(\mathbb R^2)$ with $\chi(x_0)\neq0$ be arbitrary. Using the convolution theorem and the fact that
$\widehat h=0$ almost everywhere on $\Xi_2$, we obtain
$\widehat{\chi h}(\xi)=(2\pi)^{-2}\int_{\mathbb{R}^2\setminus\Xi_2}\widehat\chi(\xi-\eta)\widehat h(\eta)\,d\eta$.
For $\xi\in\Xi_1$ and $\eta\notin\Xi_2$, the separation of the cones implies $|\xi-\eta|\geq c(|\xi|+|\eta|)$ for
some $c>0$. Since $\widehat\chi$ is rapidly decreasing, this gives
$|\widehat\chi(\xi-\eta)|\leq C_N(1+|\xi|)^{-N}w(\eta)$ for every $N\in\mathbb{N}$, with the square-integrable weight
$w(\eta)=(1+|\eta|)^{-2}$. By the Cauchy--Schwarz inequality and Plancherel's theorem, $|\widehat{\chi h}(\xi)|\leq
C_N'(1+|\xi|)^{-N}\|w\|_{L^2}\|h\|_{L^2}$ for all $\xi\in\Xi_1$. Thus, $\widehat{\chi h}$ decays rapidly in $\Xi_1$,
implying $(x_0,\xi_0)\notin\WF(h)$.
\end{proof}

\begin{corollary}[Microlocally consistent null-space reconstructions]
\label{cor:nsn-microlocal}
Let $f_\theta^{\mathrm{NSN}}=\id+P_{\ker(\mathcal R_{\acqangles})}U_\theta$ be a null-space network as in
Definition~\ref{def:nsn}, which is well defined since $\ker(\mathcal R_{\acqangles})$ is closed, and
let $z\in D(\mathcal R_{\acqangles})$. Then $f_\theta^{\mathrm{NSN}}(z)\in D(\mathcal R_{\acqangles})$,
$\mathcal R_{\acqangles} f_\theta^{\mathrm{NSN}}(z)=\mathcal R_{\acqangles} z$, and $f_\theta^{\mathrm{NSN}}(z)$ and $z$ are
microlocally consistent on $\mathbb{R}^2\times V_{\acqangles}$.
\end{corollary}

\begin{proof}
Since $f_\theta^{\mathrm{NSN}}(z)-z\in\ker(\mathcal R_{\acqangles})\mathrel{\subseteq} D(\mathcal R_{\acqangles})$,
$f_\theta^{\mathrm{NSN}}(z)\in D(\mathcal R_{\acqangles})$ and the Radon data are unchanged.
Proposition~\ref{prop:dc-microlocal} gives the microlocal conclusion.
\end{proof}

\begin{remark}[Data proximity and microlocal consistency]
Unlike exact data consistency, data proximity does not prevent additional singularities, even in the
visible region. To see this, let $u\in C_c^\infty(\mathbb{R}^2)$ and
$v\in D(\mathcal R_{\acqangles})$ with $\WF(v)\cap(\mathbb{R}^2\times V_{\acqangles})\neq\varnothing$.
For $\varepsilon>0$, consider $u_\varepsilon=u+\varepsilon v$. Then
$\|\mathcal R_{\acqangles}(u_\varepsilon-u)\|_{L^2} =\varepsilon\|\mathcal R_{\acqangles} v\|_{L^2}\to0$,
whereas $\WF(u_\varepsilon)=\WF(v)$ for every $\varepsilon>0$.
Thus, arbitrarily small data changes can introduce additional visible singularities,
and data proximity alone does not guarantee microlocal consistency.
Nevertheless, data proximity provides quantitative control of visible reconstruction errors.
Recall that $P_{\ker(\mathcal R_{\acqangles})^\perp}=\mathcal F^{-1}\chi_{V_{\acqangles}}\mathcal F$ is the Fourier
projection onto the visible frequency wedge (Remark~\ref{rem:fbp-mpi}).
With the homogeneous Sobolev norm $\|v\|_{\dot H^{-1/2}}^2\coloneqq(2\pi)^{-2}\int_{\mathbb R^2}|\widehat v(\xi)|^2
|\xi|^{-1}\,d\xi$, the Fourier-slice theorem, Plancherel's theorem in the detector variable, and polar coordinates
give, for all $h\in D(\mathcal R_{\acqangles})$,
\begin{equation}\label{eq:visible-stability}
	\|\mathcal R_{\acqangles} h\|_{L^2}=(2\pi)^{1/2}\,\|P_{\ker(\mathcal R_{\acqangles})^\perp} h\|_{\dot H^{-1/2}}.
\end{equation}
Since $(1+|\xi|^2)^{-1/2}\leq|\xi|^{-1}$, this identity implies the Sobolev stability estimate in $H^{-1/2}$
of \cite[Theorem~3.30]{frikel2013thesis} and \cite[Theorem~3.1]{wang2025robustness}. In particular, the data
change of a data-proximal network is, up to the factor $(2\pi)^{1/2}$, the change of the visible component in
$\dot H^{-1/2}$, so the budget bounds precisely this change.
Now let $u\in D(\mathcal R_{\acqangles})$, let $\|\ydata^\delta-\mathcal R_{\acqangles} u\|_{L^2}\leq\delta$, and let
$u_\delta\in D(\mathcal R_{\acqangles})$ be a reconstruction with $\|\mathcal R_{\acqangles} u_\delta-\ydata^\delta\|_{L^2}
\leq\tau\delta^r$ for some $r\in(0,1]$ and $\tau>0$ independent of $\delta$, as in
Definition~\ref{def:rates}(b). By \eqref{eq:visible-stability} and the triangle inequality,
$\|P_{\ker(\mathcal R_{\acqangles})^\perp}(u_\delta-u)\|_{\dot H^{-1/2}}\leq (2\pi)^{-1/2}(\tau+1)\delta^r$ for $0<\delta\leq1$.
Thus, a rate-$r$ data-proximity bound yields an $\mathcal O(\delta^r)$ bound on the visible component of the
reconstruction error in $\dot H^{-1/2}$.
\end{remark}

\section{Numerical experiments}
\label{sec:application}
We study limited-angle tomography to compare the reconstruction quality and
the data change of residual (RES), null-space (NSN), and
data-proximal networks with a shared correction (DP1) or with separate corrections for the null space and
its complement (DP2).
Each method is combined with Tikhonov and total-variation (TV)
initializations. We compare these methods on 500 test objects,
together with checks of the numerical projections and of the budget bound.

\subsection{Discretization}
\label{sec:ct-data}
We use the setting of Section~\ref{sec:radon-model} with $\acqangles=[-\pi/3,\pi/3]$, that is,
$A=\mathcal R_{\acqangles}$, and return to the notation $x,y$ of Sections~\ref{sec:background}
and~\ref{sec:DP}. The simulated attenuation functions are supported in $\Omega=[-1,1]^2$, which also defines the
pixel domain below; as in Section~\ref{sec:radon-model}, they are considered as elements of
$X=L^2(\mathbb R^2)$ rather than $L^2(\Omega)$. Matrices and coordinate vectors are written in bold, and
$\matid_k$ denotes the $k\times k$ identity matrix. Table~\ref{tab:ct-discretization} summarizes the
discrete function spaces and coordinate representations introduced below, together with the corresponding forward
operators.

\begin{table}[!htb]
\centering
\caption{Continuous model, discrete function spaces, and coordinate
representations in orthonormal bases.}
\label{tab:ct-discretization}
\small
\begin{tabular}{@{}llll@{}}
\toprule
Quantity & Continuous model & Discrete functions & Coordinate representation \\
\midrule
Image & $x\in X$ & $x_h\in X_h\mathrel{\subseteq} X$ & $\mathbf x\in\mathbb R^n$ \\
Data & $y\in Y$ & $y_h\in Y_h\mathrel{\subseteq} Y$ & $\mathbf y\in\mathbb R^m$ \\
Forward map & $A=\mathcal R_{\acqangles}$ & $A_h:X_h\to Y_h$ & $\mathbf A:\mathbb R^n\to\mathbb R^m$ \\
Image norm & $\|x\|_X$ & $\|x_h\|_X$ & $\|\mathbf x\|_2$ \\
Data residual & $\|Ax-y\|_Y$ & $\|A_hx_h-y_h\|_Y$ & $\|\mathbf A\mathbf x-\mathbf y\|_2$ \\
\bottomrule
\end{tabular}
\end{table}

\subsubsection*{Discrete function spaces:}
We partition $\Omega$ into $n=256^2$ square pixels $C_j$ of width $h=2/256$.
A discrete image is a piecewise constant function \(x_h\in X_h \coloneqq 
\operatorname{span}\{e_j=h^{-1}\chi_{C_j}:j=1,\ldots,n\}  \mathrel{\subseteq} X. \)
Measurements use 120 equally spaced angles in $[-\pi/3,\pi/3)$ and 363 detector positions in
$[-\sqrt2,\sqrt2]$. Relative to the $256\times256$ pixel grid, the data are thus also angularly sparse, so that
the test problem combines limited-angle and sparse-view undersampling. Let $\Delta\varphi$ and $\Delta s$ denote their spacings,
and let $D_i$ be the associated angular--detector cells of area
$w=\Delta\varphi\,\Delta s$. Discrete sinograms are represented in
\(Y_h \coloneqq  \operatorname{span}\{\psi_i=w^{-1/2}\chi_{D_i}:i=1,\ldots,m\}
 \mathrel{\subseteq} Y\), where $m=120\cdot363$. Both bases are orthonormal. Every $x_h\in X_h$ and $y_h\in Y_h$ is uniquely
determined by its coefficient vector $\mathbf x\in\mathbb R^n$ and $\mathbf y\in\mathbb R^m$, with
$x_h=\sum_j(\mathbf x)_j e_j$ and $y_h=\sum_i(\mathbf y)_i\psi_i$, and the coefficient maps
$\mathbf x\mapsto x_h$ and $\mathbf y\mapsto y_h$ are isometries, that is, $\|x_h\|_X=\|\mathbf x\|_2$ and
$\|y_h\|_Y=\|\mathbf y\|_2$. In what follows, we work with these coefficient vectors.

\subsubsection*{Discrete forward operator:}
The discrete Radon operator $A_h:X_h\to Y_h$ is defined by ray tracing, and $\mathbf A\in\mathbb R^{m\times n}$
denotes its matrix with respect to the orthonormal bases, that is, $A_hx_h=\sum_i(\mathbf A\mathbf x)_i\psi_i$. The
entry $(\mathbf A)_{ij}$ is the intersection length of the ray associated with $D_i$ and the pixel $C_j$, scaled
by $\sqrt w/h$ according to the normalization of the basis functions. Thus, $A_h$ uses ray samples rather than
exact cell averages of $Ax_h$, and $\|A_hx_h-y_h\|_Y=\|\mathbf A\mathbf x-\mathbf y\|_2$. Moreover,
$\mathbf A^\top$ represents the adjoint $A_h^*$.
The implementation stores only nonzero matrix entries and uses sparse matrix and 
transpose products.
The exact kernel $\ker(\mathbf A)$ corresponds to $\ker(A_h)$ under the
coefficient map, but it is not a discretization of $\ker(\mathcal R_{\acqangles})$, since all elements of $X_h$ are
supported in $\Omega$ and hence $\ker(\mathcal R_{\acqangles})\cap X_h=\{0\}$ (Remark~\ref{rem:radon-compact-support}).

\subsubsection*{Numerical projections:}
The networks in Section~\ref{sec:ct-networks} require the orthogonal projections onto $\ker(\mathbf A)$ and its
complement. Computing them via $\mathbf A^\mpi$ would require a dense $n\times m$ matrix and the inversion of small
singular values. Instead, we compute a matrix $\mathbf Q_p\in\mathbb R^{n\times p}$ whose orthonormal columns span the
numerical row space of $\mathbf A$, that is, the span of the right singular vectors with singular values above the
relative threshold $10^{-12}$; here $p=38597$ is the numerical rank determined by the threshold. Since the columns of $\mathbf Q_p$ are orthonormal,
$\mathbf Q_p^\top\mathbf Q_p=\matid_p$. Hence $\mathbf Q_p\mathbf Q_p^\top$ is a numerical approximation of the orthogonal projection
onto $\ker(\mathbf A)^\perp$, and $\matid_n-\mathbf Q_p\mathbf Q_p^\top$ one of the projection onto $\ker(\mathbf A)$. To reduce memory, $\mathbf Q_p$ is computed from two blockwise SVDs, with the second block
orthogonalized against the first and the threshold applied relative to the largest singular value of each block; $\mathbf Q_p$ is stored, and the projections are applied by multiplication with
$\mathbf Q_p^\top$ and $\mathbf Q_p$. The numerical kernel, that is, the range of $\matid_n-\mathbf Q_p\mathbf Q_p^\top$, contains $\ker(\mathbf A)$ and
has dimension $n-p=26939$, of
which $n-m=21976$ are due to $n>m$ alone. Since the numerical kernel also contains directions with singular
values below the threshold, a correction in it may change the data slightly; we refer to this change as leakage
(Section~\ref{sec:ct-networks}). Numerical kernels for different ranks $p$ may differ considerably in $X$
\cite{Wedin1972}, whereas directions with small singular values change the data at most by these
singular values times their norm; in the data space, the choice of the threshold therefore acts like a data-proximal budget.

\subsection{Test problem}
\label{sec:ct-synthetic}
\subsubsection*{Objects and data:}
Synthetic objects are sums of 3--12 randomly positioned and rotated
ellipses contained in the disk of radius $0.9$. Major
semiaxes range from $0.06$ to $0.6$, and axis ratios from $0.2$ to $1$. Nonnegative random
amplitudes are scaled so that the resulting images take values in $[0,1]$. Noise-free target images,
with coefficient vectors $\mathbf x_i$,
approximate pixel averages of the continuous phantoms using a $4\times4$
midpoint quadrature within each pixel. The noise-free data $\mathbf y_i\in\mathbb R^m$ are computed from analytic ellipse integrals and sampled
on the detector cells. They are deliberately not computed with $\mathbf A$, which avoids the inverse crime of using
the same discrete forward operator for data generation and reconstruction. We add noise according to $\mathbf y_i^\delta=\mathbf y_i+\delta_i\boldsymbol{\zeta}_i/
\|\boldsymbol{\zeta}_i\|_2$ with $\boldsymbol{\zeta}_i\sim\mathcal N(\mathbf 0,\matid_m)$ and $\delta_i=0.05\|\mathbf
y_i\|_2$, so that $\|\mathbf y_i^\delta-\mathbf y_i\|_2=\delta_i$ exactly, in accordance with the deterministic noise
model of Sections~\ref{sec:background} and~\ref{sec:DP}.
We distinguish the noise level $\delta_i$ from the
discretization discrepancy
$\|\mathbf y_i-\mathbf A\mathbf x_i\|_2$ and the effective noise level
$\varepsilon_i\coloneqq\|\mathbf y_i^\delta-\mathbf A\mathbf x_i\|_2$.

The main study uses 2,500 training, 500 validation, and 500 test objects, generated independently of the data
used during method development; the training objects are indexed by $i=1,\ldots,N$ with $N=2500$. Objects are assigned to splits before noise generation, and each object has one
fixed noise realization, used throughout training and evaluation.

\subsubsection*{Initial reconstructions:}
The first initialization is Tikhonov regularization (Example~\ref{ex:mpi based regularization}),
$\mathbf z_{\rm Tik}=\mathbf B_{\rm Tik}(\mathbf y^\delta)\coloneqq(\mathbf A^\top\mathbf A+\alpha\matid_n)^{-1}\mathbf A^\top\mathbf y^\delta$,
computed by conjugate gradients on the normal equations. The second
initialization is TV regularization (Example~\ref{ex:var-reg} with $\pen=\mathrm{TV}$), given by
\[
 \mathbf z_{\rm TV}=\mathbf B_{\rm TV}(\mathbf y^\delta)\in\argmin_{\mathbf z\in\mathbb R^n}
 \Bigl\{\tfrac12\|\mathbf A\mathbf z-\mathbf y^\delta\|_2^2
       +\alpha_{\rm TV}\sum_{j=1}^n\|(\mathbf D\mathbf z)_j\|_2\Bigr\}.
\]
Here $(\mathbf D\mathbf z)_j\in\mathbb R^2$ contains the forward differences of $\mathbf z$ in both image
directions at the pixel $C_j$, with zero terminal differences, and the Euclidean norm in $\mathbb R^2$ makes the TV term
isotropic. Because of the scaling of the coefficients, the sum approximates the total variation without an additional
factor~$h$. This TV minimization problem is solved by
the Chambolle--Pock algorithm~\cite{Chambolle2011}, without a positivity
constraint. The regularization parameters $\alpha=10^{-2}$ and $\alpha_{\rm TV}=10^{-3}$ were chosen manually
and are held fixed across all objects, network architectures, and evaluations; the resulting mean residuals
are $0.75\delta$ for Tikhonov and $2.18\delta$ for TV (Table~\ref{tab:ct-control}).

\subsection{Learning setup}
\label{sec:ct-networks}
All methods
share the same objects, measurements, and initial reconstructions, and test objects do not enter gradient updates.

\subsubsection*{Network architectures:}
Each correction is a five-level U-Net adapted from
Ronneberger et al.~\cite{ronneberger2015}, with four downsampling and four
bilinear upsampling steps, concatenated skip connections, and 19
convolutional layers. Two $3\times3$ convolutions with ReLU activations
per block are followed by a final linear $1\times1$ output layer; there
is no batch normalization or dropout. A single U-Net has approximately
17.3 million parameters; RES, NSN, and DP1 use one U-Net, and DP2 uses two.

Let $U_\theta,V_\theta:\mathbb R^n\to\mathbb R^n$ denote the
U-Net corrections, with $\theta$ collecting all trainable parameters.
For an input $\mathbf z$, set
$\mathbf u=U_\theta(\mathbf z)$ and $\mathbf v=V_\theta(\mathbf z)$, and, for a budget $\beta>0$, define the clipping
factors
\begin{equation}
 t_{1,\beta}=\phi_\beta(\|\mathbf A\mathbf u\|_2),\qquad
 t_{2,\beta}=\phi_\beta(\|\mathbf A\mathbf Q_p\mathbf Q_p^\top\mathbf v\|_2),
 \label{eq:ct-control-factors}
\end{equation}
with the radial clipping $\phi_\beta$ of Example~\ref{ex:scalar-dp}; as for $\mathbf u$ and
$\mathbf v$, we suppress their dependence on $\theta$ and $\mathbf z$.
The four reconstruction maps are
\begin{align}
 f^{\rm RES}_\theta(\mathbf z)&=\mathbf z+\mathbf u,\label{eq:ct-res}\\
 f^{\rm NSN}_\theta(\mathbf z)&=\mathbf z+(\matid_n-\mathbf Q_p\mathbf Q_p^\top)\mathbf u,\label{eq:ct-nsn}\\
 f^{\rm DP1}_{\theta,\beta}(\mathbf z)&=\mathbf z+(\matid_n-\mathbf Q_p\mathbf Q_p^\top)\mathbf u+t_{1,\beta}\mathbf Q_p\mathbf Q_p^\top\mathbf u,\label{eq:ct-dp1}\\
 f^{\rm DP2}_{\theta,\beta}(\mathbf z)&=\mathbf z+(\matid_n-\mathbf Q_p\mathbf Q_p^\top)\mathbf u+t_{2,\beta}\mathbf Q_p\mathbf Q_p^\top\mathbf v.\label{eq:ct-dp2}
\end{align}
DP1 corresponds to Definition~\ref{def:data-prox} with $V_\theta=U_\theta$, so the output of a single U-Net is
split into a null-space and a complementary part, whereas RES adds this output without the split. DP2 uses two
independent U-Nets, $U_\theta$ for the null space and $V_\theta$ for its complement. Separate corrections allow
different learned priors for the null-space part, which is determined by prior information alone, and for the
complementary part, which is also constrained by the data.
With the exact projection, $\|\mathbf A\mathbf P_{\ker(\mathbf A)^\perp}\mathbf v\|_2=\|\mathbf A\mathbf v\|_2$,
so both factors realize the radial clipping of Example~\ref{ex:scalar-dp}. In DP1, $\mathbf A\mathbf u$ is the data change of the
entire shared correction, so its clipping also damps the leakage of the null-space part. In DP2, only $\mathbf Q_p\mathbf Q_p^\top\mathbf v$ enters
the reconstruction, so the factor is based on the data change of this part. The two choices differ only in how the
leakage enters the data-change bound below; since the leakage is at the level of rounding errors, DP1 and DP2
differ in practice only in whether the correction is shared. 

\subsubsection*{Training objective:}
For each of the four architectures \eqref{eq:ct-res}--\eqref{eq:ct-dp2} and each initialization, we train a
separate model on the $N$ training pairs $(\mathbf y_i^\delta,\mathbf x_i)$ of noisy data and noise-free targets
(Section~\ref{sec:ct-synthetic}). The budget enters only the clipping factors and is
not a network weight. We set it to the noise level, $\beta=\delta$, which satisfies the condition $\beta^\star=\mathcal O(\delta^r)$ of
Corollary~\ref{cor:fixed-theta}\ref{item:fixed-rates} with $r=1$; since $\delta_i$
varies between objects, the $i$th object is reconstructed with budget $\delta_i$ in training and evaluation, and
$\delta_i$ is known because the noise is simulated. The weights minimize the empirical risk of
Section~\ref{sec:intro-learned},
\begin{equation}
 \risk_N(\theta)=\frac1N\sum_{i=1}^N\|f^k_{\theta,\delta_i}(\mathbf B(\mathbf y_i^\delta))-\mathbf x_i\|_2^2,
 \label{eq:ct-training-risk}
\end{equation}
where $k\in\{\mathrm{RES},\mathrm{NSN},\mathrm{DP1},\mathrm{DP2}\}$ denotes the architecture and $\mathbf
B\in\{\mathbf B_{\rm Tik},\mathbf B_{\rm TV}\}$ the initialization. Up to a constant factor, the risk is the mean
squared error of the pixel values. Each model is thus trained with
a range of budgets rather than a single one, unlike the regularizing networks of \cite{Schwab2020}, which are trained
separately for each regularization parameter. For RES and NSN, the budget index is omitted, and for DP2, both U-Nets are
optimized jointly. Initial reconstructions remain fixed, the clipping factors are evaluated from the current
corrections, and no additional data-fidelity, TV, or SSIM (structural similarity index) term enters the loss. Inputs are neither augmented nor
normalized, and output values are not clipped. With exact projections, the trained DP1 and DP2 are fixed networks with radial clipping in the sense of
Section~\ref{sec:fixed}, so that the bound of Lemma~\ref{lem:dp-proximity} on the data change applies to them
for every input.
\subsubsection*{Experimental protocol:}
\label{sec:ct-evaluation}
We minimize \eqref{eq:ct-training-risk} with Adam~\cite{kingma2015adam} (learning rate $10^{-4}$, halved every ten epochs, batch
size 2, gradient norm clipped at 1) for 60 epochs and with two training seeds, which vary the network initialization and the order of the
training data. The weights are stored after each epoch, and the final weights are selected on the validation objects
by the lowest mean MSE (mean squared error); a selection by the highest mean SSIM \cite{wang2004image} is compared in
Section~\ref{sec:ct-results}.
All models are evaluated on the same 500 test objects, with the unmodified Tikhonov and
TV reconstructions as baselines. We report MSE, PSNR (peak signal-to-noise ratio, for intensity range one), and SSIM of the pixel values, averaged
per object over the two seeds and summarized as mean $\pm$ sample standard deviation across test objects. Test results
at epoch 30 were inspected before training was extended to 60 epochs.

Besides the image quality, the results below report, for the output $f(\mathbf z)$ of a network $f$ applied to the initial reconstruction $\mathbf z$, the
data residual $\|\mathbf Af(\mathbf z)-\mathbf y^\delta\|_2$
and the data change $d=\|\mathbf A(f(\mathbf z)-\mathbf z)\|_2$ caused by the network, both in units of
$\delta$. With exact projections, Lemma~\ref{lem:dp-proximity} gives $d\leq\beta$ for DP1 and
DP2, and NSN gives $d=0$. With the numerical projections, the null-space correction changes the data by the leakage
$\ell=\|\mathbf A(\matid_n-\mathbf Q_p\mathbf Q_p^\top)\mathbf u\|_2$ (Section~\ref{sec:ct-data}). Hence NSN gives
$d=\ell$, and for DP1 and DP2 the bound increases by at most $\ell$, to $d\leq\beta+(1-t_{1,\beta})\ell$ for DP1 and
$d\leq\beta+\ell$ for DP2.
We check these bounds up to rounding, with relative tolerance $10^{-10}$. The budget $\beta=\delta_i$
bounds the change caused by the network, not the effective noise level $\varepsilon_i$, which also contains the
discretization discrepancy (Section~\ref{sec:ct-synthetic}).

Code, configurations and per-object results for both selection rules and seeds are available as described in the data availability statement.

\subsection{Results}
\label{sec:ct-results}

We report results for the four network architectures in \eqref{eq:ct-res}--\eqref{eq:ct-dp2}, each trained
separately with the Tikhonov and TV initial reconstructions and the synthetic data of
Section~\ref{sec:ct-synthetic}, using the objective \eqref{eq:ct-training-risk} and the optimization protocol of
Section~\ref{sec:ct-networks}. The following evaluation uses the same 500 test objects for all methods and both
initializations.

\begin{remark}[Expected ordering of the architectures]
\label{rem:ct-ordering}
Because all test objects follow the training distribution, the architectures are compared through the
expected squared error $\mathbb E\,\|f(\mathbf z)-\mathbf x\|_2^2$ over targets and noise, of which the risk
\eqref{eq:ct-training-risk} and the test errors are sample versions. As classes of maps, the data-proximal networks form a subclass of
the residual networks, and the null-space networks a subclass of the data-proximal networks. Restricting the class can only increase the
smallest attainable expected error. For ideally trained networks, RES is therefore expected to be the most accurate
and NSN the least accurate architecture, and the DP networks attain the accuracy of RES when the optimal correction
respects the budget. The advantage of data proximity lies instead in the bound on the data change, which is
guaranteed for every input, including inputs outside the training distribution, where learned corrections are
least reliable. The
experiments below show that the DP networks substantially improve on NSN and that, with Tikhonov initialization,
the best of them comes close to RES. Data-proximal networks thus combine much of the accuracy of residual networks
with the convergence guarantees of Section~\ref{sec:DP} and a guaranteed bound on the data change.
\end{remark}

\begin{table}[!htb]
\centering
\caption{Test errors in image space for MSE-selected checkpoints. MSE, PSNR, and SSIM are
reported as mean $\pm$ sample standard deviation across 500 test objects. For learned methods, each object's metric
is first averaged over the two training seeds (BASE does not depend on the seed). Thus the standard
deviation describes variation across objects only, not across seeds.}
\label{tab:ct-quality}
\begingroup\small\setlength{\tabcolsep}{12pt}\renewcommand{\arraystretch}{1.1}
\begin{tabular}{@{}lrrr@{}}
\toprule
Method & MSE $[10^{-4}]$ & PSNR [dB] & SSIM \\ \midrule
\multicolumn{4}{@{}l}{\textbf{Tikhonov initialization}} \\ \cmidrule(r){1-4}
BASE & $12.125\pm12.236$ & $30.63\pm3.50$ & $0.5322\pm0.1555$ \\
RES & $0.362\pm0.444$ & $46.70\pm4.56$ & $0.9896\pm0.0074$ \\
NSN & $7.590\pm8.487$ & $32.95\pm3.91$ & $0.5422\pm0.1649$ \\
DP1 & $0.369\pm0.437$ & $46.54\pm4.49$ & $0.9891\pm0.0077$ \\
DP2 & $0.443\pm0.488$ & $45.46\pm4.14$ & $0.9523\pm0.0379$ \\
\midrule
\multicolumn{4}{@{}l}{\textbf{TV initialization}} \\ \cmidrule(r){1-4}
BASE & $5.492\pm4.298$ & $33.29\pm2.23$ & $0.8697\pm0.0480$ \\
RES & $1.441\pm1.134$ & $39.37\pm2.90$ & $0.9646\pm0.0174$ \\
NSN & $3.916\pm2.175$ & $34.51\pm1.85$ & $0.8184\pm0.0611$ \\
DP1 & $1.754\pm1.183$ & $38.23\pm2.38$ & $0.9266\pm0.0359$ \\
DP2 & $1.902\pm1.294$ & $37.88\pm2.35$ & $0.8939\pm0.0395$ \\
\bottomrule
\end{tabular}\endgroup
\end{table}

\subsubsection*{Image-space errors:}
Table~\ref{tab:ct-quality} is consistent with the ordering RES, DP, NSN expected from
Remark~\ref{rem:ct-ordering} in MSE for both initializations, with BASE last, and PSNR and SSIM order the methods
in the same way; the only exception is the SSIM of NSN, which falls below that of BASE with TV initialization.
Selecting the checkpoints by validation SSIM instead of validation MSE changes individual values but
preserves this ordering and the observations below (see the data availability statement).
With Tikhonov initialization, DP1 is
practically as accurate as RES and retains $99.9\%$ of the reduction of the mean MSE that RES achieves relative to BASE, whereas NSN
removes only about a third of the mean MSE of BASE. With TV initialization, DP1 retains $92\%$ of this reduction,
and the data-space quantities below explain the larger gap. DP2 is less accurate than DP1, particularly in
SSIM, but remains substantially better than NSN. It contains DP1 as the special case $V_\theta=U_\theta$ (up to the negligible
leakage) but in our experiments provides no advantage over DP1, although it has twice as many trainable parameters and needs about
twice the training and inference time. Although TV gives the more accurate initial reconstruction,
RES, DP1, and DP2 are all more accurate with Tikhonov initialization. The TV
reconstructions are overregularized for the chosen parameter, with mean data residual $2.18\delta$ (Table~\ref{tab:ct-control}).

\begin{table}[!htb]
\centering
\caption{Data-space quantities for MSE-selected checkpoints on the 500 test objects. The mean data change
$\|\mathbf A(f(\mathbf z)-\mathbf z)\|_2$ caused by the network $f$ and the mean data residual
$\|\mathbf A f(\mathbf z)-\mathbf y^\delta\|_2$ are averaged as in Table~\ref{tab:ct-quality}; the largest data change
and the number of clipped objects in the last two columns are given as seed 0 / seed 1. All quantities are in units of the object's noise
level $\delta$, and the DP budget is $\beta=\delta$. The last column counts
objects for which the data-changing correction is reduced ($t_{1,\beta}<1$ for DP1 or $t_{2,\beta}<1$ for DP2;
\eqref{eq:ct-control-factors}). For BASE, $f(\mathbf z)=\mathbf z$, so the data change vanishes. Dashes
denote inapplicable quantities; maxima are rounded for display.}
\label{tab:ct-control}
\begingroup\small\setlength{\tabcolsep}{8pt}\renewcommand{\arraystretch}{1.1}
\begin{tabular}{@{}lrr@{\hspace{18pt}}rr@{}}
\toprule
 & \multicolumn{2}{c}{Mean $[\delta]$} & & \\ \cmidrule(r{12pt}){2-3}
Method & Change & Residual & Max. change $[\delta]$ & Clipped \\ \midrule
\multicolumn{5}{@{}l}{\textbf{Tikhonov initialization}} \\ \cmidrule(r){1-5}
BASE & -- & 0.746 & -- & -- \\
RES & 0.699 & 1.031 & 1.129\enspace/\enspace 1.095 & -- \\
NSN & $3.37\times10^{-14}$ & 0.746 & $1.25\times10^{-13}$\enspace/\enspace $7.01\times10^{-14}$ & -- \\
DP1 & 0.694 & 1.030 & 1.000\enspace/\enspace 1.000 & 2\enspace/\enspace 2 \\
DP2 & 1.000 & 1.252 & 1.000\enspace/\enspace 1.000 & 500\enspace/\enspace 500 \\
\midrule
\multicolumn{5}{@{}l}{\textbf{TV initialization}} \\ \cmidrule(r){1-5}
BASE & -- & 2.184 & -- & -- \\
RES & 1.738 & 1.309 & 8.081\enspace/\enspace 9.562 & -- \\
NSN & $1.69\times10^{-14}$ & 2.184 & $1.11\times10^{-13}$\enspace/\enspace $6.71\times10^{-14}$ & -- \\
DP1 & 0.897 & 1.692 & 1.000\enspace/\enspace 1.000 & 258\enspace/\enspace 292 \\
DP2 & 1.000 & 1.716 & 1.000\enspace/\enspace 1.000 & 500\enspace/\enspace 500 \\
\bottomrule
\end{tabular}\endgroup
\end{table}

\subsubsection*{Data-space quantities:}
The first two columns of Table~\ref{tab:ct-control} give the mean data change and data residual. With
Tikhonov initialization, DP1 and RES change the data by similar amounts and reach nearly the same mean data
residual, although
the initial reconstruction has a smaller residual. Data proximity therefore does not imply a reduction of the
data residual, nor does a smaller residual alone identify the better image. With TV initialization,
RES changes the data about twice as much as DP1 and moves them closer to the measurements. Since the bound applies
to each object, the last two columns give the largest data change and the number of clipped objects. Clipping is
rare for DP1 with Tikhonov initialization, frequent with TV initialization, and always active for DP2, and the
largest data change of both DP networks equals the budget; neither DP architecture violates the
leakage-adjusted bound of Section~\ref{sec:ct-evaluation} beyond the numerical tolerance specified there, for either training seed or checkpoint selection criterion. NSN changes the
data only by the leakage, whereas the data change of RES is not bounded and reaches $9.56\delta$ with TV
initialization (Figure~\ref{fig:ct-data-space}(a)).
The budget also explains why the gap between DP1 and RES is larger with TV initialization, since DP1
coincides with RES whenever it is not clipped. By the triangle inequality, reaching the target $\mathbf x$ from $\mathbf z$ requires the data
change $\|\mathbf A(\mathbf x-\mathbf z)\|_2\geq\|\mathbf A\mathbf z-\mathbf y^\delta\|_2-\varepsilon$. On average,
this lower bound is negative with Tikhonov initialization but exceeds the budget with TV initialization
($1.17\delta$), consistent with the frequent clipping of DP1 there. A budget adapted to the residual of the initial
reconstruction, or a more data-consistent initialization, could reduce this gap.

\begin{figure}[!htb]
\centering
\subfloat[]{
\includegraphics{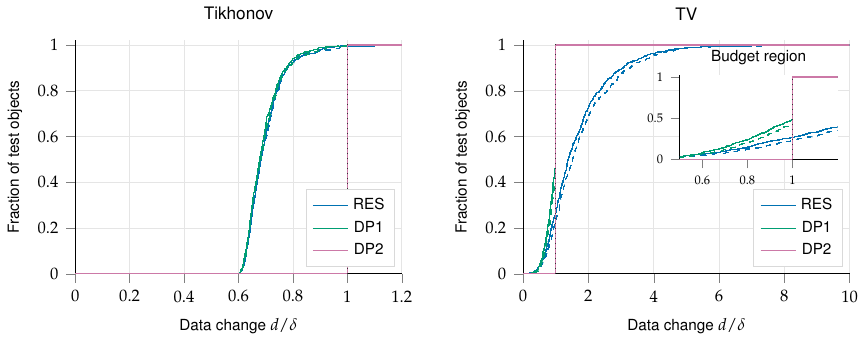}
}\\[1ex]
\subfloat[]{
\includegraphics{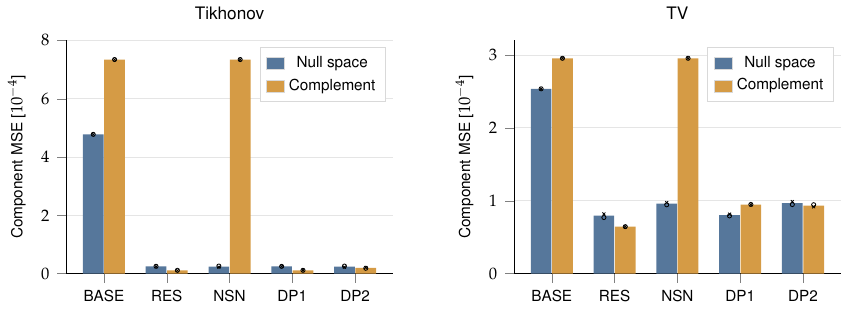}
}
\caption{(a) Empirical distribution functions of the data change in units of $\delta$ (horizontal axis $d/\delta$)
for MSE-selected checkpoints. Each curve gives the fraction of the 500 test objects
whose data change does not exceed the value on the horizontal axis. Solid and dashed curves denote seeds 0
and 1. The dotted line marks the budget $\beta=\delta$ of the DP networks, which is only a reference for RES;
a vertical jump at this line means that many objects reach the budget, as for DP2 with all objects. The inset in the
TV panel enlarges the region around the budget, while the main panel shows all RES values. NSN is omitted because
its data changes are below $1.3\times10^{-13}\delta$. (b) Mean numerical null-space and complement errors (Null space and Complement) 
for MSE-selected checkpoints. Bars average the two seed-specific means; circles and crosses mark seeds 0 and 1. The panels use
different vertical scales.}
\label{fig:ct-data-space}
\end{figure}

\subsubsection*{Null-space and complement errors:}
To see where the learned corrections act, we split the error $\mathbf e=f(\mathbf z)-\mathbf x$ into its
null-space component $(\matid_n-\mathbf Q_p\mathbf Q_p^\top)\mathbf e$ and its complement component
$\mathbf Q_p\mathbf Q_p^\top\mathbf e$, whose mean squared values add up to the MSE
(Figure~\ref{fig:ct-data-space}(b)). For the exact projection $\mathbf P=\mathbf P_{\ker(\mathbf A)}$, the identity
$(\matid_n-\mathbf P)(\mathbf z+\mathbf P\mathbf u-\mathbf x)=(\matid_n-\mathbf P)(\mathbf z-\mathbf x)$ shows that a
null-space correction leaves the complement error of any initialization unchanged, as in the error decomposition of
Keltsch et al.~\cite{keltschAdversarialStability} for pseudoinverse initialization. Accordingly, NSN reduces only the
null-space error of BASE, whereas DP1 also reduces the complement error, by a factor of about 60 with Tikhonov and 3
with TV initialization, which accounts for most of its gain over NSN. With TV initialization, the remaining gap
between DP1 and RES also lies mainly in the complement, where the budget restricts the correction.
In the discrete setting, the numerical null space takes the role of the invisible part in
Section~\ref{sec:microlocal}. It collects the directions that the data do not determine, due to both the limited
angular range and the finite number of measurements (Section~\ref{sec:ct-data}). NSN thus leaves the
complement component of the initial reconstruction, including its artifacts, unchanged, in analogy to
Corollary~\ref{cor:nsn-microlocal}, whereas the DP networks can change it, with an effect on the data bounded by the
budget.

\subsubsection*{Visual example:}
Figure~\ref{fig:ct-reconstructions} shows the ground truth and the reconstructions of a single test object for both
initializations, together with their absolute errors and central zooms. The test object was selected for spatial detail 
rather than reconstruction performance. With Tikhonov initialization, RES and DP1 appear very similar at the displayed 
scale, whereas with TV initialization, the zooms and error maps reveal more pronounced residual artifacts for DP1 and 
DP2 than for RES. NSN improves the elliptical shapes but retains noise-like or structured artifacts from the initial 
reconstruction, in line with the error decomposition above.

\begin{figure}[!htb]
\centering

\includegraphics{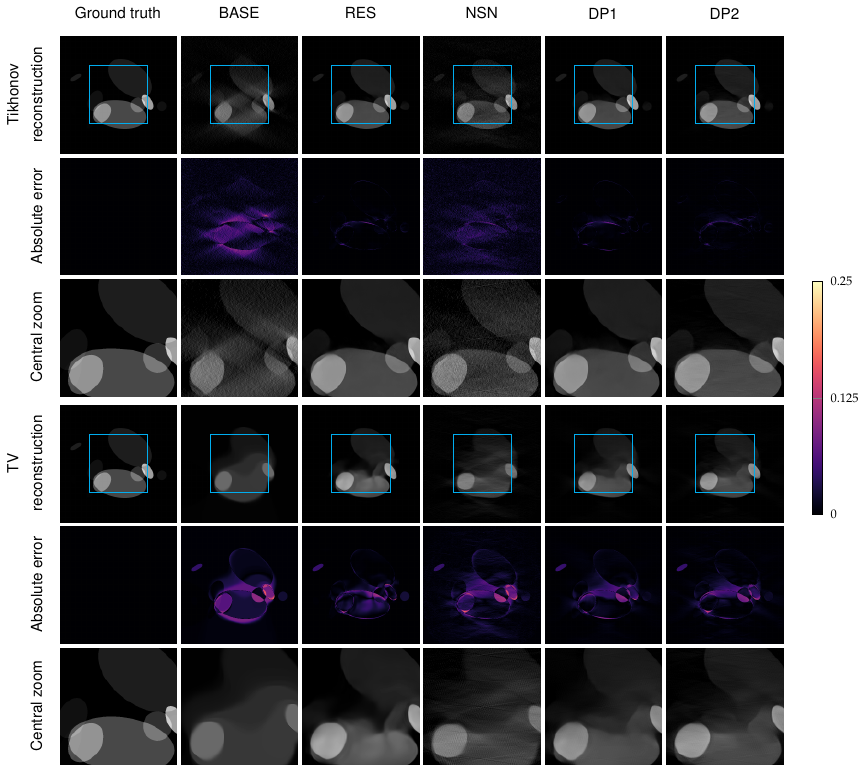}

\caption{Ground truth and reconstructions of a test object (selected for spatial
detail rather than reconstruction performance), using MSE-selected
checkpoints (seed 0). The upper and lower blocks show Tikhonov and TV
initialization, respectively. Each block displays reconstructions,
absolute errors, and central $128\times128$ pixel zooms marked in cyan.
Common display ranges are $[0,0.4]$ for images and $[0,0.25]$ for errors;
images are shown in grayscale and errors with the magma color map
(color bar).}
\label{fig:ct-reconstructions}
\end{figure}

\FloatBarrier
\section{Conclusion}
\label{sec:conclusion}

We have introduced data-proximal networks as an extension of null-space
networks within a data-driven regularization framework. This architecture allows learned 
corrections in the orthogonal complement of the null space while explicitly bounding their effect on the data. Under the stated assumptions, its composition with a potentially 
nonlinear baseline regularization converges to a learned solution selection and inherits
quantitative data-proximity bounds. The calmness analysis also gives
convergence and convergence rates for fixed networks with radial clipping. For the continuous limited-angle Radon transform, exact data consistency preserves
the visible singularities of the initial reconstruction, whereas data proximity controls the visible error in a weak
Sobolev norm that is proportional to the data norm, so that the budget bounds the change of the visible
component.

In limited-angle CT, the numerical experiments are consistent with the ordering expected for test data from the training
distribution (Remark~\ref{rem:ct-ordering}). Both DP architectures substantially improve on NSN for both
initializations, because NSN leaves the component of the initial reconstruction error in the orthogonal complement of the null
space unchanged, that is, the part determined by the data. DP1 retains $99.9\%$ of the reduction of the mean squared
error achieved by residual networks with Tikhonov initialization and $92\%$ with TV initialization, while the bound
on the data change holds for every test object. The remaining gap with TV initialization arises because the
TV reconstructions, overregularized for the chosen parameter, are far from the data, so that the target often lies outside the budget. In our experiments, DP2
offers no advantage over the simpler and faster DP1.

The present work provides a starting point for several further developments. The experiments use test data
from the training distribution, on which unconstrained residual networks are expected to be at least as accurate.
Since the bound on the data change holds for every input, inputs outside the training distribution are a natural
next test case. Measuring the reconstruction errors separately in the visible frequency wedge and its
complement would further connect the experiments with the microlocal analysis of Section~\ref{sec:microlocal}. Robustness is a related question. Keltsch et al.~\cite{keltschAdversarialStability} show that, for
null-space networks with pseudoinverse initialization, adversarial errors separate into a component determined by
the inverse operator and a learned null-space component whose sensitivity is not controlled by data consistency
alone. Data-proximal networks can modify the former component within the budget, and analyzing the sensitivity of
their learned corrections is a natural next step. The budget $\beta=\delta$ uses the known noise level and does not
take the data residual of the initial reconstruction into account. Data-proximal networks therefore benefit from an initial reconstruction whose data residual is close to the noise level, and adapting the budget to
this residual is a promising extension. This differs from regularizing networks~\cite{Schwab2020}, for which an
overregularized filter-based initialization can be favorable because the network may restore the filtered components
without a bound on the data change. Finally, the theory shows that data-proximal corrections do not compromise
convergence, and a quantitative analysis of their gain at a fixed noise level is an interesting open problem, as are
proximity constraints on other features of the reconstruction, such as selected inner products, wavelet
coefficients, or edges.

\section*{Acknowledgments}
The contribution by S.G. was part of a project that received funding from the European Union’s Horizon 2020 research and innovation programme under the Marie Sk\l{}odowska-Curie grant agreement No 847476. The views and opinions expressed herein do not necessarily reflect those of the European Commission.

\section*{Declaration of generative AI use}
During the revision and extension of an existing manuscript,
the authors used ChatGPT and Codex (OpenAI), with GPT-6 Astra Pro, 
and Claude (Anthropic), with Claude Opus 5.5 and Claude Fable 5.1.
These tools assisted with language editing, restructuring,
notation consistency, literature research, and the formulation
and review of mathematical statements and proofs.
AI assistance also supported the inspection, correction, and
extension of existing research code, the refinement of experimental
design, and the implementation, execution, analysis, and documentation
of numerical experiments. This included workflows for synthetic
data generation, training, evaluation, and experiment logging.
Computational checks included solver residuals, projection diagnostics,
data-budget tests, and consistency checks between recorded metrics,
tables, and figures.
The authors critically reviewed and, where necessary, revised
the AI-assisted text and code, checked the mathematical arguments
and cited sources, and validated the computational results.
They take full responsibility for the accuracy, originality,
and integrity of the manuscript.

\section*{Data availability}
The code, configurations and exported per-object results that support the findings of this study are
openly available at \url{https://github.com/jfrikel/data_proximal_null_space_networks_code} and are archived at
Zenodo under DOI \href{https://zenodo.org/records/24080044}{10.5281/zenodo.24080044}. The code repository
documents the evaluation protocol and the artifact requirements for reproducing the reported per-object metrics.

\end{document}